\documentclass{amsart}
\usepackage[utf8]{inputenc}
\usepackage{amsthm}
\usepackage{amsmath}
\usepackage{xcolor}

\usepackage[a4paper,  margin=1in]{geometry}

\usepackage{epic}
\usepackage{amssymb}
\usepackage{graphicx}
\usepackage{epsfig}
\newtheorem{thm}{Theorem}[section]
\newtheorem{prop}[thm]{Proposition}
\newtheorem{lem}[thm]{Lemma}
\newtheorem{cor}[thm]{Corollary}
\theoremstyle{definition}
\newtheorem{dfn}[thm]{Definition}
\theoremstyle{remark}
\newtheorem{rmk}[thm]{Remark}

\makeatletter

\numberwithin{equation}{section}
\numberwithin{figure}{section}

\usepackage[all]{xy}
\newcommand{\vertexbn}{\begin{picture}(20,20)
  \put(6,-2){\line(1,0){20}}\put(6,2){\line(1,0){20}}
  \put(20,0){\line(-1,1){10}}\put(20,0){\line(-1,-1){10}}
  \end{picture}}
  \newcommand{\vertexbbn}{\begin{picture}(20,20)
 \put(6,0){\line(1,0){19}}\put(6,0){\line(1,0){19}}
  \put(5,-3){\line(1,0){21}}\put(5,3){\line(1,0){21}}
  \put(20,0){\line(-1,1){10}}\put(20,0){\line(-1,-1){10}}
  \end{picture}}
   \newcommand{\vertexbbbn}{\begin{picture}(20,20)
\put(6,-1){\line(1,0){20}}\put(6,1){\line(1,0){20}}
  \put(5.5,-3){\line(1,0){21}}\put(5.5,3){\line(1,0){21}}
  \put(20,0){\line(-1,1){10}}\put(20,0){\line(-1,-1){10}}
  \end{picture}}
  \newcommand{\vertexcn}{\begin{picture}(20,20)
  \put(6,-2){\line(1,0){20}}\put(6,2){\line(1,0){20}}
  \put(15,0){\line(1,-1){10}}\put(15,0){\line(1,1){10}}
  \end{picture}}
  \newcommand{\vertexcccn}{\begin{picture}(20,20)
  \put(6,-1){\line(1,0){20}}\put(6,1){\line(1,0){20}}
  \put(5.5,-3){\line(1,0){20}}\put(5.5,3){\line(1,0){20}}
  \put(15,0){\line(1,1){10}}\put(15,0){\line(1,-1){10}}
  \end{picture}}
    \newcommand{\vertexccn}{\begin{picture}(20,20)
  \put(7,0){\line(1,0){20}}
  \put(6,-3){\line(1,0){22}}\put(6,3){\line(1,0){22}}
  \put(15,0){\line(1,1){10}}\put(15,0){\line(1,-1){10}}
  \end{picture}}
\makeatother
\begin{document}
\title{ Vogan diagrams of  Hyperbolic  Kac-Moody algebras  }
\author{B Ransingh and K C Pati}
\date{}
\maketitle

\begin{abstract}
 This paper constructs the Vogan diagrams for hyperbolic Kac-Moody algebras,
 which have potential physical applications in cosmological billiards.
\end{abstract}


\noindent Keywords - Vogan diagram; Kac-Moody algebras

\section{Introduction}
For a complex semi-simple Lie algebra $\mathfrak{g}$, it is well known that the conjugacy classes of real
forms of $\mathfrak{g}$ are in one to one correspondence with the conjugacy classes of involutions of $\mathfrak{g}$, if
we associate to a real form $\mathfrak{g}_{\mathbb{R}}$ one of its Cartan involutions $\theta$. Using a suited pair 
$(\mathfrak{h}, \prod)$ of a
Cartan subalgebra $\mathfrak{h}$ and a basis $\prod$ of the associated root system, an involution is described by
a "Vogan diagram"; the equivalences of Vogan diagrams describe how such a diagram changes 
when one changes $(\mathfrak{h}, \prod)$. When $\mathfrak{g}$ is now an affine Kac-Moody algebra,
the one to one correspondence $\mathfrak{g}_{\mathbb{R}}\leftrightarrow\theta$
was established by Ben Messaoud and Rousseau (Theorem 3.1  \cite{rousseu}).

The last two decades shows a gradual advancement in classification of real form of semisimple Lie algebras via Kac-Moody algebras to
Kac-Moody superalgebras by Satake diagrams and Vogan diagrams.
Satake diagram based on splits Cartan subalgebra    where as Vogan diagram based on
maximally compact Cartan subalgebra.
Kac-Moody  algebras have played an increasingly crucial role in various 
areas of mathematics as well as theoretical physics. The hyperbolic Kac-Moody 
algebras which constitute a subclass of Lorentzian Kac-Moody algebras and some
of their (almost split)  real forms have appeared in a variety of problems in the realms of string field theory 
and super gravity theories. The almost split real forms have been classified  by
 Tits-Satake diagrams \cite{Messaud:Hyperbolic}. 

Knapp \cite{Knapp:Lie groups} brought Vogan diagram  into the light to represent the real forms of the complex simple Lie algebras. 
Batra \cite{batra:affine,batra:vogan} developed a corresponding theory of Vogan diagrams for  real forms of  nontwisted
 affine Kac-Moody Lie algebras. Hyperbolic Kac-Moody algebras among all types of Kac-Moody algebras is least studied. 
However  these types of algebras have led to many potential physical applications.

In space like singularity, at each spatial point, the degree of freedom that carry the essential dynamics are the logarithms 
$\beta^{i}\equiv -\ln a_{i}$ of the scale factors $a_i$ along a set of (special) independent 
spatial directions ($ds^{2}=-dt^{2}+\sum _{i}a_{i}^{2}(t,x)(w^{i})^{2}$). The billiard analysis refers to the dynamics, in the vicinity of a
spacelike singularity of a gravitational model described by the
Lagrangian\cite{Damaour}
\begin{eqnarray*} {\mathcal L}_D &=& ^{(D)}R\sqrt{\vert^{(D)}g\vert}\,dx^0\wedge...\wedge dx^{D-1} -\sum_\alpha \star
d\phi^\alpha\wedge d\phi^\alpha - \nonumber \\ & \, &
-\frac{1}{2}\sum_p\,e^{\sum_\alpha \lambda^{(p)}_\alpha \phi^\alpha}\star F^{(p+1)}\wedge
F^{(p+1)}, \quad D\geq 3\end{eqnarray*} where R is the spatial curvature scalar and $\star$ is the Hodge star operator.

The real parameter $\lambda_{\alpha}^{(p)}$ measures the strength of the coupling to the dilaton. Field strength is $F^{(p)}$
The dilatons are denoted by $\phi^\alpha, (\alpha=1,...,N)$; their kinetic terms are normalized with a weight 1 with respect to the
Ricci scalar. The Einstein metric has Lorentz signature $(-,+,...,+)$; its determinant is $^{(D)}g$.

The light-like wall and the  walls bounding the  billiard have different origins:
some arise from the Einstein-Hilbert action and involve only the
scale factors $\beta^i,\,(i=1,...,d)$, introduced through the
Iwasawa decomposition of the space metric.
For example a criterion for gravitational dynamics to be chaotic is that the billiard has finite volume. 
This in turn stems from the remarkable property that the billiard can be identified with
the fundamental Weyl chamber of an hyperbolic Kac-Moody algebra. In this paper we have obtained 
the Vogan diagrams of some of the hyperbolic Kac-Moody algebras which have been discussed in paper \cite{thesis:mtheory1},
so that these diagrams can lead to more interesting studies related to billiard, M-theory etc.
More particularly we have also obtained Vogan diagram of hyperbolic Kac-Moody
algebra $E_{10}$.


\section{Preliminaries}
The following definitions are useful tools in our work.\\
\begin{dfn}({\it Generalised Cartan Matrix} ({\it GCM})) \cite{kac:edition} An integral matrix
$A=(a_{ij})_{i,j=1}^{r}$ is said to be {\it GCM} if  

\[ a_{ij}=\begin{cases}

\begin{array}{cc}
2 & i=j\\
\leq0 & i\neq j
\end{array}
\end{cases}
\]

\noindent and

\[a_{ij}=0\Rightarrow a_{ji}=0\]\\
\end{dfn}
\begin{dfn}
({\it Kac-Moody Algebra})\cite{kac:edition} Let $e_i$, $f_i$ and $h_i$ denote  the $3r$  Chevalley 
generators. The {\it Kac-Moody  algebra}  is defined as the  algebra $\mathfrak{g}$ together with the following  relations
\end{dfn}
\begin{itemize}
 \item[(a)]   $\left[h_{i}, h_{j}\right]=0$,
 \item[(b)]   $\left[e_{i},f_{i}\right]=h_{i}$,
 \item[(c)]   $\left[e_{i},f_{j}\right]=0$ if $i\neq j$,
 \item[(d)]   $\left[h_{i},e_{j}\right]=a_{ij}e_{j}$,
 \item[(e)]   $\left[h_{i},f_{j}\right]=-a_{ij}f_{j}$,
 \item[(f)]   $(\mbox{ad }e_{i})^{1-a_{ij}}e_{j}=0$,
 \item[(g)]   $(\mbox{ad }f_{i})^{1-a_{ij}}f_{j}=0$ if
$i\neq j$.
\end{itemize}
\vspace{0.5cm}

\begin{dfn} ({\it Dynkin Diagram}) \cite{kac:edition} A Dynkin diagram associated with a $GCM$ $A$  
is a graph with the following properties:
\begin{itemize}
 \item [(a)] The Dynkin diagram has $r$ vertices.
 \item[(b)] When $a_{ij}a_{ji}=n\leq4$, the vertices $i$ and $j$ are joined by $n$ lines.
 \item [(c)] if $a_{ij}a_{ji}\leq 4$ and $|a_{ij}|\geq |a_{ji}|$, the vertices $i$ and $j$ are connected by   $|a_{ij}|$ lines
and these lines are equipped with an arrow pointing towards $i$ if  $|a_{ij}|>1$.
 \item [(d)] If $a_{ij}a_{ji}>4$. the vertices $i$ and $j$ are connected by a bold-faced line equipped with an ordered pair of integers
 $|a_{ij}|$, $|a_{ji}|$.
 \item [(e)] If $n>4$, $i$ and $j$ are joined by a thick line on which we write $|A_{ij}|,|A_{ji}|$ with $|A_{ij}|\geq |A_{ji}|$ . This 
 $n>4$ case will only concern us when we discuss rank-2 hyperbolic  algebras.
\end{itemize}\end{dfn}
\begin{rmk}
It is clear that $A$ is indcomposable if and only if Dynkin diagram is a connected graph.
\end{rmk}
The the relationship between a $GCM$, a {\it Dynkin diagram}  and an {\it algebra} as in
 above relations is one to one.\\

\begin{dfn} ({\it Hyperbolic Kac-Moody Algebra}) \cite{kac:edition} A generlised Cartan matrix (GCM) $A$ is called a
{\it matrix of hyperbolic type} if it is indecomposable, symmetrizable of indefinite type and if every proper connected subdiagram of
Dynkin diagram of $A$ is of finite or affine type. The Kac-Moody algebra is called {\it hyperbolic Kac-Moody Algebra} if the GCM is of
hyperbolic type. \end{dfn}

\begin{dfn}  ({\it Dynkin diagram of hyperbolic Kac-Moody Algebras}) The general strategy in searching for hyperbolic
Dynkin diagrams of rank $r+1$ is as follows: \end{dfn}
\begin{enumerate}
 \item[(i)] Draw all possible Lie and / or affine (including semisimple) diagrams of rank $r$
 \item [(ii)]Add an extra root, trying all possible lengths
 \item [(iii)]Try connecting the new root to the old ones in all ways consistent with a symmetrizable $GCM$.
 \item [(iv)]Test the resulting diagram by removing any point to see whether it reduces to (perhaps a disconnected combinations of) known
finite or affine algebras, the twisted ones being included among the latter. A diagram that survives the test is of the hyperbolic
type.
\end{enumerate}
\section{Root system and relative root system}





Let $\omega'$ is the standard Cartan semi-involution, 
$\omega$ Chevalley Cartan involution  $\mathfrak{g}$, s.t. $w(h)=-h$, and $w(e_{i})=-f_{i}$ and $\tau$ is a involutive diagram automorphism.
Let $\sigma'$ be a semi-involution of the first kind and $\mathfrak{g}_\mathbb{R}$ be the corresponding
almost compact real form. and thus $\sigma=\sigma'\omega'=\omega'\sigma'$ is the Cartan involution of
$\mathfrak{g}_\mathbb{R}$. Hence $\mathfrak{h}_{\mathbb{R}}$ is a $\sigma$-stable maximally split Cartan
subalgebra of $\mathfrak{g}_\mathbb{R}$ and $\mathfrak{t}_{\mathbb{R}}={\mathfrak{h}_{\mathbb{R}}^{-\sigma}}$
is a $\sigma$ stable maximal compact toral subalgebra. The group $K={G_{\mathbb{R}}}^\sigma$ is transitive
on the set of $\sigma$ stable maximally compact Cartan subalgebra  (respectively $\sigma$-stable maximally 
compact toral subalgebra) of $\mathfrak{g}_\mathbb{R}$. The root space decomposition becomes
$$\mathfrak{g}=\mathfrak{h}\oplus\underset{\alpha\in\triangle}{\bigoplus}\mathfrak{g}_{\alpha}$$
$$Z_{\mathfrak{g}}(\mathfrak{t}_{0})=\mathfrak{h}\oplus\underset{\alpha\in\triangle_{real}}{\bigoplus}\mathfrak{g}_{\alpha}$$
and $\mathfrak{k}_{0}\cap Z_\mathfrak{g}(\mathfrak{t}_{0})=\mathfrak{k}_{0}\cap \mathfrak{h}=\mathfrak{t}_{0}$. Therefore 
$\mathfrak{t}_{0}$ is maximal abelian in $\mathfrak{k}_{0}$ and $\mathfrak{h}_{0}$ is maximally compact.

 We get the relative or restricted root system for $\alpha\in\triangle(\mathfrak{g},\mathfrak{h})$ by denoting
 $\alpha'=\alpha|\mathfrak{t}_\mathbb{R}$ the restriction of  $\alpha$ to $\mathfrak{t}_{\mathbb{R}}$.
 The (infinite) relative (or restricted) root system $\triangle'=
 \triangle(\mathfrak{g}_{\mathbb{R}},\mathfrak{t}_{\mathbb{R}})
=\{\alpha';\alpha\in\triangle(\mathfrak{g},\mathfrak{h})\}$\textbackslash{}$\{0\}$ and 

$$\mathfrak{h}_{\mathbb{R}}=\mathfrak{h}\cap \mathfrak{g}^{\sigma'}=\mathfrak{h}\cap \mathfrak{g}_\mathbb{R}$$

$$\mathfrak{t}_{\mathbb{R}}=\underset{i\in\Omega}{\bigoplus}\mathbb{R}(\alpha^{\vee}_{i}+\alpha^{\vee}_{\tau_{i}})$$
where $\Omega$ is the simple root corresponding to maximally compact Cartan subalgebras.

\subsection{The relative Kac-Moody matrix and  Dynkin diagram}

 A relative Kac-Moody matrix $B'=(\alpha_{j}',\alpha_{i}'^{\vee})=(b'_{i,j})$ which satisfies 
$b'_{i,j}\in\mathbb{Z}$, $b'_{i,i}\leq2$, $b'_{i,j}\in\mathbb{Z}^{-}$ for $i\neq j$, $b'_{i,j}=0$ iff $b'_{j,i}=0$.
The relative Kac-Moody matrix is associated a graph $S(B')$, with $|I'|$ vertices, called the relative Dynkin diagram as follows;
we associate each $i$ a vertex equipped with a cross if $b'_{i,i}=1$, with a $-$ sign if $b'_{i,i}<0$ and with $0$ if $b'_{i,i}=0$.
Two vertices $i$ and $j$ are linked iff $b'_{i,j}<0$; if $b'_{i,i}$ and $b'_{j,j}$ are positive and $a'_{i,j}a'_{j,i}=n_{i,j}\leq 4$,
the vertices $i$ and $j$ are joined by $m_{i,j}=$max $(|a'_{i,j}||a'_{j,i}|)$ line(s) with an arrow pointing towards
the vertex $i$ if $|a'_{i,j}|>1$. If $n_{i,j}\geq5$ or one of the two coefficient  $b'_{i,i}$ or $b'_{j,j}$ is non-positive,
the vertices $i$ and $j$ are joined by by a thick line on which we write $|a'_{i,j}|$ (besides the vertex $i$ ) and 
$|a'_{j,i}|$ (beside the vertex j). We use the same nomenclature as in \cite{Messaud:Hyperbolic} namely $Z_{-}$ stands
for $\bullet_{-}$

\section{Real form of  hyperbolic Kac-Moody algebras and  Vogan diagram }

We now  recall some terms associated with Kac-Moody algebras.\\

A real form of $\mathfrak{g}$ is an algebra $\mathfrak{g}_{\mathbb{R}}$ over $\mathbb{R}$ such that there exists an isomorphism from 
$\mathfrak{g}$ to $\mathfrak{g}_{\mathbb{R}}\otimes\mathbb{C}$. If we replace $\mathbb{C}$ with $\mathbb{R}$ in the definition of 
$\mathfrak{g}$, we obtain a real form $\mathfrak{g}_{\mathbb{R}}$ which is called split. A real form of $\mathfrak{g}$ correponds
to a semi-linear involution of $\mathfrak{g}$. A semi-linear involution is an automorphism $\tau$ of $\mathfrak{g}$ such that $\tau^{2}
=Id$ and $\tau(\lambda x=\overline{\lambda}\tau(x))$ for $\lambda\in\mathbb{C}$. A Borel subalgebra of $\mathfrak{g}$ is a 
maximal completely solvable subalgebra. There are two types of Borel subalgebra (positive or
negative). A linear or semi-linear automorphism of $\mathfrak{g}$ is said to be of the first kind if it transforms a Borel subalgebra
in $\mathfrak{g}$  to a Borel subalgebra of the same sign. A linear or semi-linear automorphism of $\mathfrak{g}$
is said to be of the second kind if it transforms a Borel subalgebra into a Borel subalgebra of the oppposite sign.

Let $\mathfrak{g}_{\mathbb{R}}$ be a real form of $\mathfrak{g}$ and fix an isomorphism from $\mathfrak{g}$ to
 $\mathfrak{g}_{\mathbb{R}\otimes\mathbb{C}}$. Then the Galois group $\Gamma = Gal(\mathbb{C}/\mathbb{R})$ acts on $\mathfrak{g}$
and the corresponding group $G$. A real form is the fixed point set $\mathfrak{g}^\Gamma$. If $\Gamma$ consists of first kind of
automorphism then $\mathfrak{g}_{\mathbb{R}}$  is almost split, otherwise if the non-trivial element of $\Gamma$ is a second kind 
automorphism the $\mathfrak{g}_{\mathbb{R}}$ is almost compact.
 We consider 
 \begin{enumerate}
 \item The semi-involutions $\sigma'$, of the second kind of $\mathfrak{g}$.
 \item The involution $\theta$, of the first kind of  $\mathfrak{g}$.
 \end{enumerate}
\begin{thm} \cite{rousseu}  The relation $\sigma\approx\theta$ if and only if
 \begin{enumerate}
 \item[({\it a})] $\omega'=\theta\sigma'=\sigma'\theta$ is a Cartan semi-involution.
 \item [({\it b})] $\theta$ and $\sigma'$ stabilize the same Cartan subalgebra $\mathfrak{h}$.
 \item [({\it c})] $\mathfrak{h}$ is contained in a minimal $\sigma'$-stable positive parabolic subalgebra.
\end{enumerate}
\end{thm}
\begin{rmk} The above theorem is proved in (\cite{rousseu2} {Corollary 7.7}) for affine Kac-Moody algebras.
Since the minimal parabolic subalgebra $\mathfrak{p}$ ,
satisy $\sigma'(\mathfrak{p})\cap \mathfrak{p}=\mathfrak{h}$.
The condition $(c)$ in the theorem is equivalent to require that $\mathfrak{h}^{\sigma'}$ is a maximally
compact Cartan subalgebras. Then this relation induces a bijection between the conjugacy classes under 
$Aut(\mathfrak{g})$ of semi-involutions of the second 
kind and conjugacy classes of involutions of the first kind.\end{rmk}

\section{ Vogan diagram}
Let $\mathfrak{g}_\mathbb{R}$ be an almost compact real form of $\mathfrak{g}$. Let $\sigma$ be the Cartan involution and 
$\mathfrak{g}_\mathbb{R}=\mathfrak{k}_{0}\oplus\mathfrak{p}_0$ be the correspoding Cartan decomposition. Let $\mathfrak{h}_{0}=
\mathfrak{t}_{0}\oplus\mathfrak{a}_{0}$ be a maximally compact $\sigma$-stable Cartan subalgebra of  $\mathfrak{g}_\mathbb{R}$,
with complexification $\mathfrak{h}=\mathfrak{t}\oplus\mathfrak{a}$

The roots of $(\mathfrak{g},\mathfrak{h})$ are imaginary on $\mathfrak{t}_0$ and real on $\mathfrak{a}_0$. A root is real 
if it takes real values on $\mathfrak{a}_0$.
A root is real if it takes real values  on $\mathfrak{h}_0$ (i.e., vanishes on $\mathfrak{t}_0$), imaginary if it takes 
purely imaginary values on  $\mathfrak{h}_0$ (i.e., vanishes on $\mathfrak{a}_0$) and  complex otherwise.

For any root $\alpha$, $\sigma\alpha$ is the root $\sigma\alpha (H)=\alpha(\sigma ^{-1}H)$. If $\alpha$ is imaginary, then
$\sigma\alpha=\alpha$. Thus $\mathfrak{g}_\alpha$ is $\sigma$-stable, and we have 
$\mathfrak{g}_\alpha=(\mathfrak{g}_\alpha\cap\mathfrak{t})\oplus(\mathfrak{g}_\alpha\cap\mathfrak{p})$.
Since $\mathfrak{g}_\alpha$ is one
dimensional, $\mathfrak{g}_\alpha\subseteq\mathfrak{t}$ or $\mathfrak{g}_\alpha\subseteq\mathfrak{p}$. 
We call an imaginary root $\alpha$ 
compact if    $\mathfrak{g}_\alpha\subseteq\mathfrak{t}$ , noncompact if $\mathfrak{g}_\alpha\subseteq\mathfrak{p}$.
Let $\mathfrak{h}_0$ be a $\sigma-$ stable Cartan subalgebra of $\mathfrak{g}_\mathbb{R}$. 
Then there are no real roots if and only if 
$\mathfrak{h}_0$ is maximally compact [1, Proposition 3.8].

Let $\triangle=\triangle(\mathfrak{g},\mathfrak{h})$ be the set of roots. There are no real roots, i.e.,
 no roots that vanishes on $\mathfrak{t}$. We choose a positive system $\triangle^{+}$ for $\triangle$ that takes
$i\mathfrak{t}_0$ before $\mathfrak{a}_0$. Since $\sigma$ is
$+1$ on   $\mathfrak{t}_0$ and $-1$ on $\mathfrak{a}_0$ and since there are no real roots,
$\sigma(\triangle^{+})=\triangle^{+}$. Therefore
$\sigma$ permutes the simple roots. It must fix the simple roots that are imaginary and permutes in two cycles the simple roots
that are complex.
\begin{dfn}
 An abstract Vogan diagram is said to be an abstract Dynkin diagram with additional structures as follows:
 \begin{enumerate}
  \item Automorphisms of order one or two which is  labeled by 2-element orbits.
  \item Choose a subset of 1-element orbits and paint the vertices  corresponding  member of the subset.  
 \end{enumerate}
\end{dfn}

By the Vogan diagram of the triple $(\mathfrak{g},\mathfrak{h}_{0},\triangle^{+})$, we mean the Dynkin diagram of 
of $\triangle^{+}$ with the two-element orbits under $\sigma$ ({\it diagram automorphism}) labeled by an arrow and with 
the one-element orbits painted or not,
according to the corresponding imaginary simple root, noncompact or compact.\\

Here we reformulate the Borel and de Siebenthal Theorem for hyperbolic Kac-Moody algebras of regular simply laced Dynkin diagrams
and others in general separately.\\

\begin{thm} (Borel and de Siebenthal 1)  Every equivalence class of  over extended (regular simply laced hyperbolic)
Kac-Moody algebra can be represented by a Vogan diagram with at most one painted vertex. \end{thm}
\begin{proof}
It is well known that the equivalence class of Vogan diagram of 
an extended Dynkin diagram $D$, with an extra vertex $p$ with extra edge consist of at most two painted vertices.
The over extended Dynkin diagram (for every regular  simply laced diagrams) consists of an extra vertex with an extra edge
inaddition to at most two painted vertices in affine diagrams. Thus we obtain a painting with at most three painted vertices. 
But  using the  $F(i)$ algorithm and diagram automorphism we get a painting with at most one painted vertex for every
regular simply laced hyperbolic Kac-Moody algebras.\end{proof}
 
The following Corollary is immediate.

\begin{cor}
 If a connected graph $D$ is a hyperbolic Dynkin diagram (simply laced), then
 \begin{itemize}
  \item [(a)] every painting on $D$ can be simplified via a sequence of $F<i>$ to a painting with single painted vertex.
  \item [(b)]every connected sub-graph of $D$ satisfies property $(a)$.
 \end{itemize}
\end{cor}
\begin{proof}
 The proof of the corollary is straightforward using the $F<i>$ sequences.
\end{proof}
We also construct the Borel and de Siebenthal theorem  including non-simply laced Dynkin diagrams.
\begin{thm} (Borel and de Siebenthal 2) 
Every equivalence class of  over extended (hyperbolic)
Vogan diagrams has a representative with at most three  vertex painted.\end{thm}
\begin{proof} The equivalence class of Vogan diagram of 
an extended Dynkin diagram $D$, with an extra vertex $p$ with extra edge consist of at most two painted vertices.
The over extended Dynkin diagram  consists of a over extra vertex with very extra edge,
so together with at most two painted vertices of affine diagrams we obtain a painting with at most three painted vertices.
Since in some hyperbolic non simply laced diagrams  contains three directed arrows towards short roots which 
cannot be equivalent to double or single painted diagrams by $F\langle i\rangle$ algorithm.
 \end{proof}

\subsection{ Equivalence classes of Vogan diagrams}

Batra \cite{batra:affine} defined equivalence of Vogan diagrams by defining
equivalence relations generated
by the following two operations:

(a) Application of an automorphism of the Dynkin diagram. 

(b) Change in the positive system by reflection in a simple, noncompact root,
i.e., by
a vertex which is colored in the Vogan diagram.\\

 As a consequence of reflection
by a simple,
noncompact root $\alpha$, the rule for single and triple lines is that we leave
$\alpha$
colored and its immediate neighbour is changed to the opposite color. The rule
for double
lines is that if  $\alpha$ is the smaller root, then there is no change in the
color
of its immediate neighbour, but we leave  $\alpha$  colored. If $\alpha$ is the
larger
root, then we leave $\alpha$ colored and its immediate neighbour is changed to
the opposite
color. If two Vogan diagrams are not equivalent to each other, they are called
nonequivalent. 

By labeling vertices with $1,\cdots,n$ as in  \cite{Chuah:equivalence,Chuah:quick}
a Vogan diagram with painted vertices $ i_1,\cdots,i_k$,  $1\leq i_1\leq
\cdots < i_k\leq n$, is denoted by $( i_1,\cdots,i_k)$.
Suppose $i\in {\{ i_1,\cdots,i_k\}}$, so that $i$ is a painted vertices. The $F\langle i\rangle$ 
operates on the Vogan diagram as follows. It acts on the root system by
reflection corresponding to the noncompact simple root $i$, 
as a result it leads an equivalent Vogan diagram. 
A combinatorial description for two Vogan diagrams to be equivalent with
$F\langle i\rangle$ operation is as follows
\begin{itemize}
 \item [(i)]The color of $i$ and all vertices not adjacent to $i$ remain unchanged.
 \item [(ii)]If $j$ is joined to $i$ by a double edge and $j$ 's long, the color
of $j$ remains unchanged.		
 \item [(iii)]Apart from the above exceptions, the colors of all vertices
adjacent to $i$, are reversed.
\end{itemize}

For Exceptional Dynkin diagram like $E_{10}$, we get the following Lemma from \cite{Chuah:vogan-extended}\\

\begin{lem} (\cite{Chuah:vogan-extended}, Lemma 3.8 )
\begin{itemize}
 \item [(a)] {\it For $q\geq4$ and $p=2,3$, we get $(p,q)\sim(0,p-1,q-1)$ and
$(0,p,q)\sim(p-1,q-1)$}
 \item [(b)] {\it For $q\geq4$, $(1,q)\sim(0,q-1)$ and $(0,1,q)\sim(q-1)$}
\end{itemize}\end{lem}
Applying all these technique in our hand, now we proceed to find the Vogan diagrams of some algebras which are
important in respect of physical applications
in string theory, High eneregy physics etc. We start with $E_{10}$.\\

\noindent{\bf Example-1} Dynkin diagram of hyperbolic Kac-Moody algebra $E_{10}$.\\

 The Dynkin diagram of $E_{10}$ is given as.\\
 
 
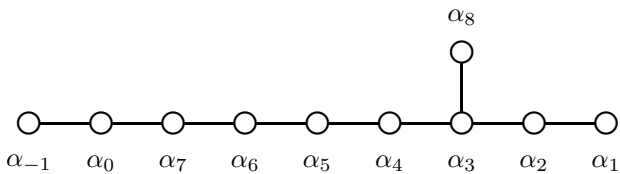
\vspace{0.5cm}
\begin{picture}(60,30) \thicklines
\put(7,0){\circle{8}}\put(34,0){\circle{8}}\put(61,0){\circle{8}}\put(88,0){\circle{8}}\put(115,0){\circle{8}}\put(142,0){\circle{8}}
\put(169,0){\circle{8}}\put(196,0){\circle{8}}\put(223,0){\circle{8}} 
\put(169,27){\circle{8}}
\put(11,0){\line(1,0){19}}\put(38,0){\line(1,0){19}}\put(65,0){\line(1,0){19}}\put(92,0){\line(1,0){19}}\put(119,0){\line(1,0){19}}
\put(146,0){\line(1,0){19}}\put(173,0){\line(1,0){19}}\put(200,0){\line(1,0){19}}
\put(169,4){\line(0,1){19}}
 \put(7,-15){\makebox(0,0){$\alpha_{-1}$}} \put(34,-15){\makebox(0,0){$\alpha_{0}$}}\put(61,-15){\makebox(0,0){$\alpha_{7}$}}
 \put(88,-15){\makebox(0,0){$\alpha_{6}$}}\put(115,-15){\makebox(0,0){$\alpha_{5}$}}\put(142,-15){\makebox(0,0){$\alpha_{4}$}}
 \put(169,-15){\makebox(0,0){$\alpha_{3}$}}\put(196,-15){\makebox(0,0){$\alpha_{2}$}}\put(223,-15){\makebox(0,0){$\alpha_{1}$}}
 \put(169,40){\makebox(0,0){$\alpha_{8}$}}
\end{picture}
\vspace{1.5cm}

$\alpha_{0}$ and $\alpha_{-1}$ roots correspond to affine and hyperbolic extension
of simple Lie algebra $E_{8}$ respectively.
Enumerating the vertices of  $E_{10}$ differently as shown below we get the Dynkin diagram and the following proposition.\\

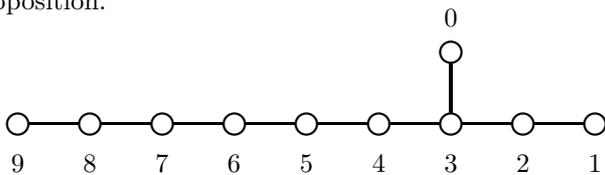
\begin{picture}(60,30) \thicklines
\put(7,0){\circle{8}}\put(34,0){\circle{8}}\put(61,0){\circle{8}}\put(88,0){\circle{8}}\put(115,0){\circle{8}}\put(142,0){\circle{8}}
\put(169,0){\circle{8}}\put(196,0){\circle{8}}\put(223,0){\circle{8}} 
\put(169,27){\circle{8}}
\put(11,0){\line(1,0){19}}\put(38,0){\line(1,0){19}}\put(65,0){\line(1,0){19}}\put(92,0){\line(1,0){19}}\put(119,0){\line(1,0){19}}
\put(146,0){\line(1,0){19}}\put(173,0){\line(1,0){19}}\put(200,0){\line(1,0){19}}
\put(169,4){\line(0,1){19}}
 \put(7,-15){\makebox(0,0){$9$}} \put(34,-15){\makebox(0,0){$8$}}\put(61,-15){\makebox(0,0){$7$}}
 \put(88,-15){\makebox(0,0){$6$}}\put(115,-15){\makebox(0,0){$5$}}\put(142,-15){\makebox(0,0){$4$}}
 \put(169,-15){\makebox(0,0){$3$}}\put(196,-15){\makebox(0,0){$2$}}\put(223,-15){\makebox(0,0){$1$}}
 \put(169,40){\makebox(0,0){$0$}}
\end{picture}
\vspace{1.5cm}

\begin{prop}
The equivalence classes of Vogan diagram of  $E_{10}$ are given by
\begin{enumerate}
\item $1\sim5\sim\left(0,4\right)\sim\left(0,9\right)\sim\left(0,8\right)$
\item
$2\sim3\sim7\sim8\sim\left(0,7\right)\sim\left(0,
6\right)\sim0\sim4\sim6\sim9\sim\left(0,3\right)\sim\left(0,1\right)\sim\left(0,
2\right)\sim\left(0,5\right)\sim\left(0,7\right)$\end{enumerate}\end{prop}
\begin{proof}
The  proposition can be proved by using  $E_{n}$ arguments and  Lemma 5.5 , switching the  sequences as follows 

\begin{center}
$\begin{array}{ccccc}
(0,7) & \sim & (1,8) & \mbox{by Lemma 5.5(b)}\\
 & \sim & (0,2,9) & \mbox{by Lemma 5.5(a)}\\
 & \sim & (8) & F\left\langle 0,3,4,5,6,7,8\right\rangle \\
 & \sim & (0,1,9) & \mbox{by Lemma 5.5(b)}\\
 & \sim & \left(2\right) & \mbox{by\ensuremath{F\left\langle
9,8,7,6,5,4,3,2\right\rangle }}\\
 & \sim & \left(0,1,4\right) & \mbox{by F\ensuremath{\left\langle
2,3,0\right\rangle }}\\
 & \sim & \left(3\right) & \mbox{by Lemma 5.5(b)}\end{array}$
\end{center}

\begin{center}
$\begin{array}{ccccc}
\left(0,5\right) & \sim & \left(1,6\right) & \mbox{by Lemma 5.5(b)}\\
 & \sim & \left(0,2,7\right) & \mbox{by Lemma 5.5(a)}\\
 & \sim & \left(6\right) & \mbox{by}F\left\langle 0,3,4,5,6\right\rangle \\
 & \sim & \left(0,1,7\right) & \mbox{by Lemma 5.5(b)}\\
 & \sim & \left(2,8\right) & \mbox{by Lemma 5.5(a)}\\
 & \sim & \left(0\right) & F\left\langle 8,7,6,5,4,3,0\right\rangle \\
 & \sim & \left(0,3\right) & \mbox{F\ensuremath{\left\langle 0\right\rangle }}\\
 & \sim & \left(1,4\right) & \mbox{by Lemma 5.5(b)}\\
 & \sim & \left(0,2,5\right)\sim\left(4\right) & F\left\langle
0,3,4\right\rangle \\
 & \sim & \left(0,1,5\right) & \mbox{by Lemma 5.5(b)}\\
 & \sim & \left(2,6\right) & \mbox{by Lemma 5.5(b)}\\
 & \sim & \left(0,7\right) & F\left\langle 6,5,4,3,0\right\rangle \\
 & \sim & \left(1,8\right) & \mbox{by Lemma 5.5(b)}\\
 & \sim & \left(0,2\right) & F\left\langle 8,7,6,5,4,3,2\right\rangle \\
 & \sim & \left(0,1\right) & F\left\langle 0,2,1\right\rangle \\
 & \sim & \left(9\right) & \mbox{by F\ensuremath{\left\langle
1,2,3,4,5,6,7,8,9\right\rangle }}\end{array}$
\end{center}

\begin{center}
$\begin{array}{cccc}
\left(5\right) & \sim & \left(0,1,6\right) & \mbox{by Lemma 5.5(b)}\\
 & \sim & \left(1,5\right) & F\left\langle 0,1,3,4,5\right\rangle \\
 & \sim & \left(0,4\right) & \mbox{by Lemma 5.5(b)}\sim\left(1\right)\mbox{by
}F\left\langle 0,3,2,1\right\rangle \\
 & \sim & \left(0,9\right) & F\left\langle 1,2,3,4,5,6,7,8,9\right\rangle
\end{array}$
\end{center}

\begin{center}
$\begin{array}{cccc}
\left(0,6\right) & \sim & \left(1,7\right) & \mbox{by Lemma 5.5(b)}\\
 & \sim & \left(0,2,8\right) & \mbox{by Lemma 5.5(a)}\\
 & \sim & \left(7\right) & \mbox{by}F\left\langle 0,3,4,5,6\right\rangle \\
 & \sim & \left(0,1,8\right) & \mbox{by Lemma 5.5(b)}\\
 & \sim & \left(2\right) & \mbox{by F\ensuremath{\left\langle
8,7,6,5,4,3,2\right\rangle }}\\
 & \sim & \left(0,1,4\right) & \mbox{by F\ensuremath{\left\langle
2,3,0\right\rangle }}\\
 & \sim & \left(3\right) & \mbox{by Lemma 5.5(b)}\end{array}$
\end{center} \end{proof}
\begin{rmk}  The two Vogan diagrams of hyperbolic Kac-Moody $E_{10}$ are given below. \\

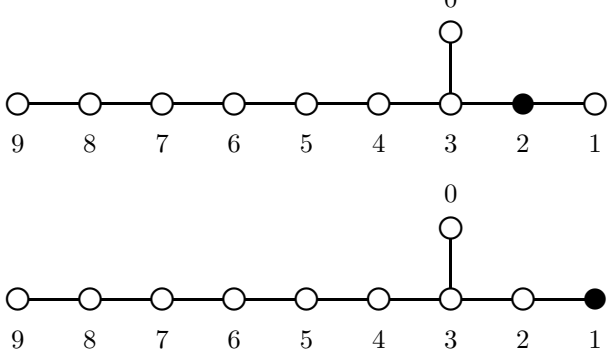
\begin{picture}(60,30) \thicklines
\put(7,0){\circle{8}}\put(34,0){\circle{8}}\put(61,0){\circle{8}}\put(88,0){\circle{8}}\put(115,0){\circle{8}}\put(142,0){\circle{8}}
\put(169,0){\circle{8}}\put(196,0){\circle*{8}}\put(223,0){\circle{8}} 
\put(169,27){\circle{8}}
\put(11,0){\line(1,0){19}}\put(38,0){\line(1,0){19}}\put(65,0){\line(1,0){19}}\put(92,0){\line(1,0){19}}\put(119,0){\line(1,0){19}}
\put(146,0){\line(1,0){19}}\put(173,0){\line(1,0){19}}\put(200,0){\line(1,0){19}}
\put(169,4){\line(0,1){19}}
 \put(7,-15){\makebox(0,0){$9$}} \put(34,-15){\makebox(0,0){$8$}}\put(61,-15){\makebox(0,0){$7$}}
 \put(88,-15){\makebox(0,0){$6$}}\put(115,-15){\makebox(0,0){$5$}}\put(142,-15){\makebox(0,0){$4$}}
 \put(169,-15){\makebox(0,0){$3$}}\put(196,-15){\makebox(0,0){$2$}}\put(223,-15){\makebox(0,0){$1$}}
 \put(169,40){\makebox(0,0){$0$}}
\end{picture}

\vspace{1.5cm}
\begin{picture}(60,30) \thicklines
\put(7,0){\circle{8}}\put(34,0){\circle{8}}\put(61,0){\circle{8}}\put(88,0){\circle{8}}\put(115,0){\circle{8}}\put(142,0){\circle{8}}
\put(169,0){\circle{8}}\put(196,0){\circle{8}}\put(223,0){\circle*{8}} 
\put(169,27){\circle{8}}
\put(11,0){\line(1,0){19}}\put(38,0){\line(1,0){19}}\put(65,0){\line(1,0){19}}\put(92,0){\line(1,0){19}}\put(119,0){\line(1,0){19}}
\put(146,0){\line(1,0){19}}\put(173,0){\line(1,0){19}}\put(200,0){\line(1,0){19}}
\put(169,4){\line(0,1){19}}
 \put(7,-15){\makebox(0,0){$9$}} \put(34,-15){\makebox(0,0){$8$}}\put(61,-15){\makebox(0,0){$7$}}
 \put(88,-15){\makebox(0,0){$6$}}\put(115,-15){\makebox(0,0){$5$}}\put(142,-15){\makebox(0,0){$4$}}
 \put(169,-15){\makebox(0,0){$3$}}\put(196,-15){\makebox(0,0){$2$}}\put(223,-15){\makebox(0,0){$1$}}
 \put(169,40){\makebox(0,0){$0$}}
\end{picture}
\vspace{1.5cm}
\end{rmk}
Now we consider some more examples.\\

 \begin{prop}
 The Vogan diagrams of $H_{116}^{(3)}, IG_{2}(1,3)$
 \begin{itemize}
  \item [(a)] $(1,3)$
  \item [(b)] $(1)$
  \item [(c)] $(2)\sim (1,2,3)$
  \item [(d)] $(3)$
  \item [(e)] $(1,2)\sim (2,3)$
 \end{itemize}

  \end{prop}

   
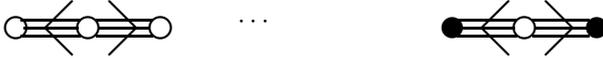
\begin{picture}(160,30) \thicklines
\put(7,0){\circle{8}}\put(34,0){\circle{8}}\put(90,0){$\cdots$}\put(61,0){\circle{8}}
\put(3,0){\vertexccn}\put(32,0){\vertexbbn}
  \end{picture}
  \begin{picture}(160,30) \thicklines
\put(7,0){\circle*{8}}\put(34,0){\circle{8}}\put(61,0){\circle*{8}}
\put(3,0){\vertexccn}\put(32,0){\vertexbbn}
  \end{picture}
 
\begin{prop}
 The equivalence classes of Vogan diagrams of $H_{35}^{4}$ are given by
 \begin{itemize}
  \item [(a)] $(1,3,4)\sim (1,2,3,4)$ by $F<4>$
  \item [(b)] $1\sim (1,2)$
  \item [(c)] $2$
  \item [(d)] $3\sim (3,2)$
  \item [(e)] $4\sim (2,4)$
  \item [(f)] $(1,3)\sim (1,2,3)$
  \item [(g)] $(1,4)\sim (1,2,4)$
  \item [(h)] $(3,4)\sim (2,3,4)$
 \end{itemize}

\end{prop}

  
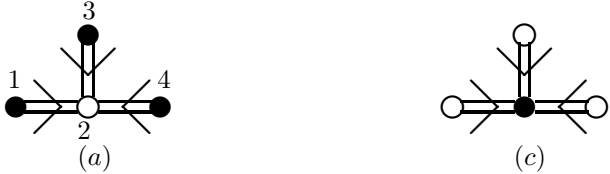
\begin{picture}(160,30) \thicklines
  \put(4,7){$1$} \put(30,-12){$2$} \put(60,7){$4$}\put(32,33){$3$}
\put(7,0){\circle*{8}}\put(34,27){\circle*{8}}\put(34,0){\circle{8}}\put(61,0){\circle*{8}}
\put(4,0){\vertexbn}\put(32,0){\vertexcn}
\put(32,4){\line(0,1){20}}\put(36,4){\line(0,1){20}}
\put(34,12){\line(-1,1){10}}\put(34,12){\line(1,1){10}}\put(30,-22){$(a)$}
  \end{picture}
  \begin{picture}(160,30) \thicklines
\put(7,0){\circle{8}}\put(34,27){\circle{8}}\put(34,0){\circle*{8}}\put(61,0){\circle{8}}
\put(4,0){\vertexbn}\put(32,0){\vertexcn}
\put(32,4){\line(0,1){20}}\put(36,4){\line(0,1){20}}
\put(34,12){\line(-1,1){10}}\put(34,12){\line(1,1){10}}\put(30,-22){$(c)$}
  \end{picture}
\vspace{1cm}

The rest diagrams for $(b),(d),(e),(f),(g),(h)$ can be constructed similarly. 

\noindent{\bf Example-2}  Dynkin diagram of a rank 4 hyperbolic algebra \\ 
Consider the Dynkin diagram
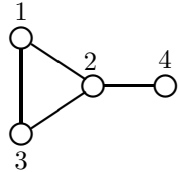
\begin{center}

\begin{picture}(120,30) \thicklines
\put(61,18){\circle{8}}\put(61,-18){\circle{8}}\put(88,0){\circle{8}}\put(115,0){\circle{8}}
\put(87,9){\makebox(0,0){2}} 
\put(85,2.5){\line(-3,2){20}}\put(85,-2.5){\line(-3,-2){20}}\put(92,0){\line(1,0){19}}\put(61,-14){\line(0,1){28}}
\put(61,28){\makebox(0,0){1}} \put(61,-28){\makebox(0,0){3}}   \put(115,10){\makebox(0,0){4}} 
\end{picture}             \end{center}

\vspace{1cm}

\begin{prop} The equivalences classes of Vogan diagram of the above hyperbolic Kac-Moody  algebra
are given by
\begin{itemize}
 \item [(a)]$1\sim  3\sim 4$ 
 \item [(b)] $2$
\end{itemize}\end{prop}
\begin{proof}
 By applying $F\langle i\rangle$

\[\begin{array}{cccc}
\left(1\right) & \sim & \left(1,2,3\right) & \mbox{ by } F\langle 1\rangle\\
 & \sim & (2,4) & \mbox{ by } F\langle 2\rangle\\
 & \sim & (4) \sim (2,4)&\sim (1,2,3) \mbox { by } F\langle 4,2\rangle\\

 & (1,2,3)&\sim (3)& \mbox{ by } F\langle 3\rangle\\

 \end{array}\]

\noindent Also we have

\[\begin{array}{cccc}
\left(2\right) & \sim & \left(1,2,3,4\right) \sim (1,4) & \mbox { by } F\left\langle 2,1\right\rangle \\
 & \left(1,2,3,4\right) \sim & \left(1,3,4\right) &\mbox { by }  F\left\langle 4\right\rangle \\
 
 & (1,2,3,4)&\sim (3,4)\sim (2,3,4)&\sim(1,2) \sim(1,3)\sim(2,3)\sim(1,2,4)\mbox{ by } F\langle 3,4,2,1,3,2\rangle\\
\end{array}\]

Vogan diagram of this algebra is given by

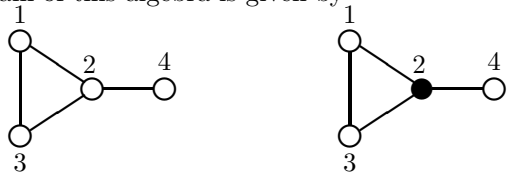
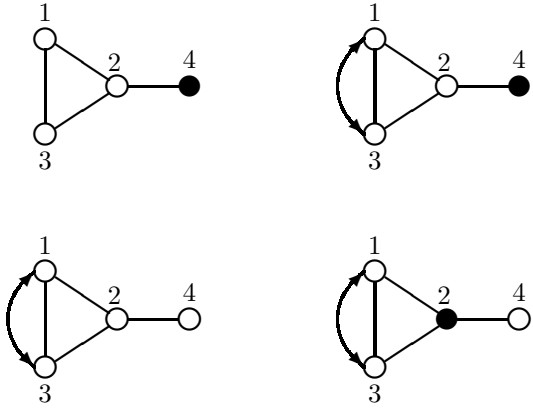
\begin{picture}(120,30) \thicklines
\put(61,18){\circle{8}}\put(61,-18){\circle{8}}\put(88,0){\circle{8}}\put(115,0){\circle{8}}
\put(87,9){\makebox(0,0){2}} 
\put(85,2.5){\line(-3,2){20}}\put(85,-2.5){\line(-3,-2){20}}\put(92,0){\line(1,0){19}}\put(61,-14){\line(0,1){28}}
\put(61,28){\makebox(0,0){1}} \put(61,-28){\makebox(0,0){3}}   \put(115,10){\makebox(0,0){4}} 
\end{picture}
\begin{picture}(120,30) \thicklines
\put(61,18){\circle{8}}\put(61,-18){\circle{8}}\put(88,0){\circle*{8}}\put(115,0){\circle{8}}
\put(87,9){\makebox(0,0){2}} 
\put(85,2.5){\line(-3,2){20}}\put(85,-2.5){\line(-3,-2){20}}\put(92,0){\line(1,0){19}}\put(61,-14){\line(0,1){28}}
\put(61,28){\makebox(0,0){1}} \put(61,-28){\makebox(0,0){3}}   \put(115,10){\makebox(0,0){4}} 
\end{picture}
\vspace{1.5cm}

\begin{picture}(120,30) \thicklines
\put(61,18){\circle{8}}\put(61,-18){\circle{8}}\put(88,0){\circle{8}}\put(115,0){\circle*{8}}
\put(87,9){\makebox(0,0){2}} 
\put(85,2.5){\line(-3,2){20}}\put(85,-2.5){\line(-3,-2){20}}\put(92,0){\line(1,0){19}}\put(61,-14){\line(0,1){28}}
\put(61,28){\makebox(0,0){1}} \put(61,-28){\makebox(0,0){3}}   \put(115,10){\makebox(0,0){4}} 
\end{picture}
\begin{picture}(120,30) \thicklines
\put(61,18){\circle{8}}\put(61,-18){\circle{8}}\put(88,0){\circle{8}}\put(115,0){\circle*{8}}
\put(87,9){\makebox(0,0){2}} 
\put(85,2.5){\line(-3,2){20}}\put(85,-2.5){\line(-3,-2){20}}\put(92,0){\line(1,0){19}}\put(61,-14){\line(0,1){28}}
\put(61,28){\makebox(0,0){1}} \put(61,-28){\makebox(0,0){3}}   \put(115,10){\makebox(0,0){4}} 
 \qbezier(47,0)(47,10)(57,18)
        \qbezier(47,0)(47,-10)(57,-18)
\put(57,18){\vector( 1, 1){0}}
\put(57,-18){\vector( 1, -1){0}}
\end{picture}

\vspace{2cm}

\begin{picture}(120,30) \thicklines
\put(61,18){\circle{8}}\put(61,-18){\circle{8}}\put(88,0){\circle{8}}\put(115,0){\circle{8}}
\put(87,9){\makebox(0,0){2}} 
\put(85,2.5){\line(-3,2){20}}\put(85,-2.5){\line(-3,-2){20}}\put(92,0){\line(1,0){19}}\put(61,-14){\line(0,1){28}}
\put(61,28){\makebox(0,0){1}} \put(61,-28){\makebox(0,0){3}}   \put(115,10){\makebox(0,0){4}} 
 \qbezier(47,0)(47,10)(57,18)
        \qbezier(47,0)(47,-10)(57,-18)
\put(57,18){\vector( 1, 1){0}}
\put(57,-18){\vector( 1, -1){0}}
\end{picture}
\begin{picture}(120,30) \thicklines
\put(61,18){\circle{8}}\put(61,-18){\circle{8}}\put(88,0){\circle*{8}}\put(115,0){\circle{8}}
\put(87,9){\makebox(0,0){2}} 
\put(85,2.5){\line(-3,2){20}}\put(85,-2.5){\line(-3,-2){20}}\put(92,0){\line(1,0){19}}\put(61,-14){\line(0,1){28}}
\put(61,28){\makebox(0,0){1}} \put(61,-28){\makebox(0,0){3}}   \put(115,10){\makebox(0,0){4}} 
 \qbezier(47,0)(47,10)(57,18)
        \qbezier(47,0)(47,-10)(57,-18)
\put(57,18){\vector( 1, 1){0}}
\put(57,-18){\vector( 1, -1){0}}
\end{picture}
\vspace{1cm}

\end{proof}

\noindent{\bf Example-3} The Dynkin diagram of another rank 4 hyperbolic Kac-Moody algebra is given by 
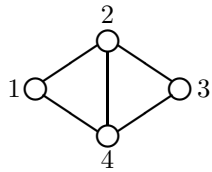
\begin{displaymath}
 \begin{picture}(100,30) \thicklines
\put(0,0){\circle{8}}\put(27,-18){\circle{8}}\put(0,-36){\circle{8}}\put(-27,-18){\circle{8}}
\put(0,-32){\line(0,1){28}}\put(4,-34){\line(3,2){20}}\put(-4,-34){\line(-3,2){20}}\put(-24,-15){\line(3,2){20}}
\put(24,-15){\line(-3,2){20}}
\put(0,10){\makebox(0,0){2}}\put(36,-18){\makebox(0,0){3}}\put(0,-45){\makebox(0,0){4}}\put(-35,-18){\makebox(0,0){1}}
\end{picture}
\end{displaymath}

\vspace{2cm}

\begin{prop}
The equivalence classes of  another rank 4 hyperbolic Kac-Moody algebra are given by

\begin{enumerate}
\item[(a)] $3 \sim 1$ (by symmetry) 
\item[(b)] $4 \sim 2$ by symmetry
\item [(c)] $(2,4)$

\end{enumerate}\end{prop}
\begin{proof}
By symmetry(diagram automorphism) and $F_i$ algorithm, we get .
\begin{itemize}
 \item [(a)] $1\sim (1,2,4)\sim(3,4)\sim (1,2,4)\sim (2,3)$ by $F\langle 1,4,4,2\rangle$
 \item [(b)] $4\sim (1,2,3,4)\sim(1,3)\sim (1,2,3,4)\sim 2$  $F\langle 4,1,1,2 \rangle$
 \item [(a*)] $(1,2)\sim (2,3,4)\sim (1,4)\sim(2,3,4)\sim 3$ by $F\langle 2,4,4,3 \rangle$
 \item [(c)] $(2,4)\sim (1,2,3)\sim(1,3,4)\sim (2,4)$ by $F\langle 2,3,4 \rangle$
 
\end{itemize}
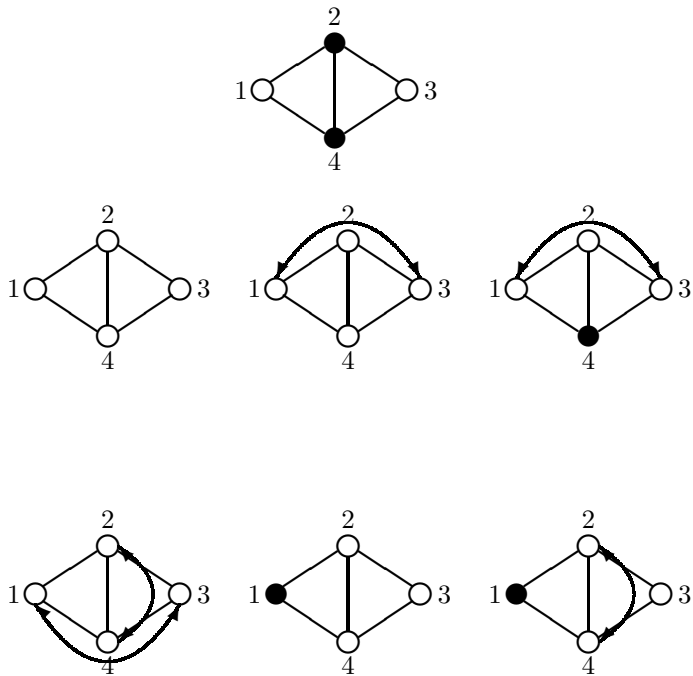
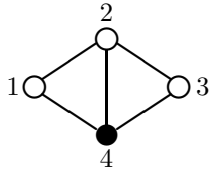
\begin{displaymath}
 \begin{picture}(100,30) \thicklines
\put(0,0){\circle*{8}}\put(27,-18){\circle{8}}\put(0,-36){\circle*{8}}\put(-27,-18){\circle{8}}
\put(0,-32){\line(0,1){28}}\put(4,-34){\line(3,2){20}}\put(-4,-34){\line(-3,2){20}}\put(-24,-15){\line(3,2){20}}
\put(24,-15){\line(-3,2){20}}
\put(0,10){\makebox(0,0){2}}\put(36,-18){\makebox(0,0){3}}\put(0,-45){\makebox(0,0){4}}\put(-35,-18){\makebox(0,0){1}}
\end{picture}
\end{displaymath}

\vspace{1cm}
\begin{displaymath}
 \begin{picture}(90,30) \thicklines
\put(0,0){\circle{8}}\put(27,-18){\circle{8}}\put(0,-36){\circle{8}}\put(-27,-18){\circle{8}}
\put(0,-32){\line(0,1){28}}\put(4,-34){\line(3,2){20}}\put(-4,-34){\line(-3,2){20}}\put(-24,-15){\line(3,2){20}}
\put(24,-15){\line(-3,2){20}}
\put(0,10){\makebox(0,0){2}}\put(36,-18){\makebox(0,0){3}}\put(0,-45){\makebox(0,0){4}}\put(-35,-18){\makebox(0,0){1}}
\end{picture}
\begin{picture}(90,30) \thicklines
\put(0,0){\circle{8}}\put(27,-18){\circle{8}}\put(0,-36){\circle{8}}\put(-27,-18){\circle{8}}
\put(0,-32){\line(0,1){28}}\put(4,-34){\line(3,2){20}}\put(-4,-34){\line(-3,2){20}}\put(-24,-15){\line(3,2){20}}
\put(24,-15){\line(-3,2){20}}

\qbezier(-27,-14)(0,28)(27,-14)

\put(-27,-14){\vector(-1,-2){0}}
\put(27,-14){\vector( 1, -2){0}}

\put(0,10){\makebox(0,0){2}}\put(36,-18){\makebox(0,0){3}}\put(0,-45){\makebox(0,0){4}}\put(-35,-18){\makebox(0,0){1}}
\end{picture}
\begin{picture}(90,30) \thicklines
\put(0,0){\circle{8}}\put(27,-18){\circle{8}}\put(0,-36){\circle*{8}}\put(-27,-18){\circle{8}}
\put(0,-32){\line(0,1){28}}\put(4,-34){\line(3,2){20}}\put(-4,-34){\line(-3,2){20}}\put(-24,-15){\line(3,2){20}}
\put(24,-15){\line(-3,2){20}}

\qbezier(-27,-14)(0,28)(27,-14)

\put(-27,-14){\vector(-1,-2){0}}
\put(27,-14){\vector( 1, -2){0}}

\put(0,10){\makebox(0,0){2}}\put(36,-18){\makebox(0,0){3}}\put(0,-45){\makebox(0,0){4}}\put(-35,-18){\makebox(0,0){1}}
\end{picture}
\end{displaymath}
\vspace{2cm}

\begin{displaymath}
\begin{picture}(90,30) \thicklines
\put(0,0){\circle{8}}\put(27,-18){\circle{8}}\put(0,-36){\circle{8}}\put(-27,-18){\circle{8}}
\put(0,-32){\line(0,1){28}}\put(4,-34){\line(3,2){20}}\put(-4,-34){\line(-3,2){20}}\put(-24,-15){\line(3,2){20}}
\put(24,-15){\line(-3,2){20}}

\qbezier(4,0)(30,-18)(4,-36)
\qbezier(-27,-22)(0,-65)(27,-22)
\put(4,0){\vector(-1,1){0}}
\put(4,-36){\vector( -1, -1){0}}
\put(-27,-22){\vector(-1,1){0}}
\put(27,-22){\vector( 1,2){0}}
\put(0,10){\makebox(0,0){2}}\put(36,-18){\makebox(0,0){3}}\put(0,-45){\makebox(0,0){4}}\put(-35,-18){\makebox(0,0){1}}

\end{picture}
\begin{picture}(90,30) \thicklines
\put(0,0){\circle{8}}\put(27,-18){\circle{8}}\put(0,-36){\circle{8}}\put(-27,-18){\circle*{8}}
\put(0,-32){\line(0,1){28}}\put(4,-34){\line(3,2){20}}\put(-4,-34){\line(-3,2){20}}\put(-24,-15){\line(3,2){20}}
\put(24,-15){\line(-3,2){20}}

\put(0,10){\makebox(0,0){2}}\put(36,-18){\makebox(0,0){3}}\put(0,-45){\makebox(0,0){4}}\put(-35,-18){\makebox(0,0){1}}

\end{picture}
\begin{picture}(90,30) \thicklines
\put(0,0){\circle{8}}\put(27,-18){\circle{8}}\put(0,-36){\circle{8}}\put(-27,-18){\circle*{8}}
\put(0,-32){\line(0,1){28}}\put(4,-34){\line(3,2){20}}\put(-4,-34){\line(-3,2){20}}\put(-24,-15){\line(3,2){20}}
\put(24,-15){\line(-3,2){20}}

\qbezier(4,0)(30,-18)(4,-36)

\put(4,0){\vector(-1,1){0}}
\put(4,-36){\vector( -1, -1){0}}
\put(0,10){\makebox(0,0){2}}\put(36,-18){\makebox(0,0){3}}\put(0,-45){\makebox(0,0){4}}\put(-35,-18){\makebox(0,0){1}}

\end{picture}\end{displaymath}

\vspace{2cm}
\begin{center}

\begin{picture}(100,30) \thicklines
\put(0,0){\circle{8}}\put(27,-18){\circle{8}}\put(0,-36){\circle*{8}}\put(-27,-18){\circle{8}}
\put(0,-32){\line(0,1){28}}\put(4,-34){\line(3,2){20}}\put(-4,-34){\line(-3,2){20}}\put(-24,-15){\line(3,2){20}}
\put(24,-15){\line(-3,2){20}}

\put(0,10){\makebox(0,0){2}}\put(36,-18){\makebox(0,0){3}}\put(0,-45){\makebox(0,0){4}}\put(-35,-18){\makebox(0,0){1}}

\end{picture}

\end{center}
\vspace{2cm}
\end{proof}

\noindent{\bf Example-4} The dynkin diagram of a rank 5 hyperbolic Kac-Moody algebra is given by

\begin{picture}(60,30) \thicklines
\put(61,18){\circle{8}}\put(61,-18){\circle{8}}\put(88,0){\circle{8}}\put(115,0){\circle{8}}
\put(142,0){\circle{8}}
\put(85,2){\line(-3,2){20}}\put(85,-2){\line(-3,-2){20}}\put(85,0){\vertexbn}
\put(119,0){\line(1,0){19}}
\put(61,28){\makebox(0,0){2}} \put(61,-28){\makebox(0,0){1}} \put(87,10){\makebox(0,0){3}}  \put(115,10){\makebox(0,0){4}} 
\put(142,10){\makebox(0,0){5}}
\end{picture} 

\vspace{1cm}

\begin{prop} The equivalenc classes of Vogan diagram of this rank 5 
hyperbolic Kac-Moody algebra are given by
\begin{itemize}
 \item[(a)] $1\sim 2 \sim 3$
 \item[(b)] $4 \sim (4,5) \sim 5$
\end{itemize}\end{prop}
\begin{proof}
 By using $F_i$ algorithm

\[\begin{array}{cccc}
2 & \sim & \left(2,3\right) & \mbox{ by }F\langle 2\rangle \\
 & \sim & \left(1,2,3,4\right) & \mbox{ by }F\langle 3\rangle \end{array}\]

\noindent and

\[\begin{array}{cccc}
1 & \sim & \left(1,3\right) & \mbox{by }F\langle 1\rangle\\
 & \sim & \left(1,2,3,4\right) & \mbox{by }F\langle 3\rangle\\
 & \sim & \left(1,2,3,4,5\right) & \mbox{by }F\langle 4\rangle\end{array}\]

\noindent from above  we get $1\sim2$\\
$3\sim (1,2,3,4)\mbox{ by } F\langle 3\rangle\sim (1,2,4)\mbox{ by } F\langle 1\rangle$\\

So $1\sim 2 \sim 3$

$\begin{array}{cccccccc}
4 & \sim\left(4,5\right) & \mbox{by }F\langle 4\rangle & \sim & 5 & \mbox{by
}F\langle 5\rangle\end{array}$

\noindent So the Vogan Diagrams of the above dynkin diagram with diagram automorphism become

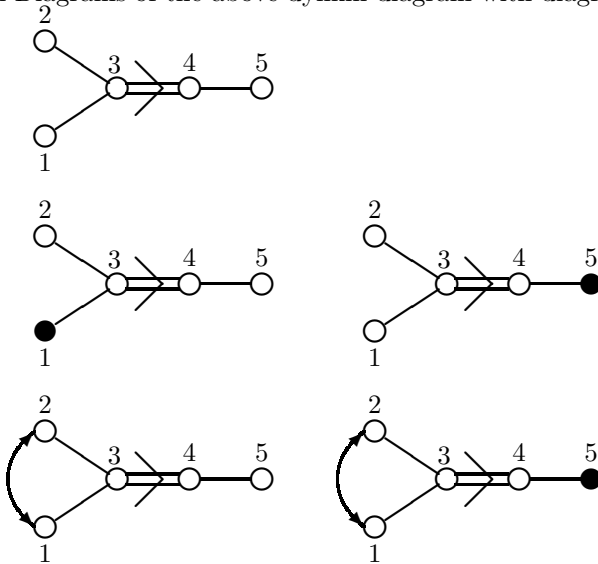
\begin{picture}(120,30) \thicklines
\put(61,18){\circle{8}}\put(61,-18){\circle{8}}\put(88,0){\circle{8}}\put(115,0){\circle{8}}
\put(142,0){\circle{8}}\put(87,9){\makebox(0,0){3}} 
\put(85,2){\line(-3,2){20}}\put(85,-2){\line(-3,-2){20}}\put(85,0){\vertexbn}
\put(119,0){\line(1,0){19}}
\put(61,28){\makebox(0,0){2}} \put(61,-28){\makebox(0,0){1}}  \put(115,10){\makebox(0,0){4}} 
\put(142,10){\makebox(0,0){5}}
\end{picture} 

\vspace{1.5cm}

\begin{picture}(120,30) \thicklines
\put(61,18){\circle{8}}\put(61,-18){\circle*{8}}\put(88,0){\circle{8}}\put(115,0){\circle{8}}
\put(142,0){\circle{8}}\put(87,9){\makebox(0,0){3}} 
\put(85,2){\line(-3,2){20}}\put(85,-2){\line(-3,-2){20}}\put(85,0){\vertexbn}
\put(119,0){\line(1,0){19}}
\put(61,28){\makebox(0,0){2}} \put(61,-28){\makebox(0,0){1}}  \put(115,10){\makebox(0,0){4}} 
\put(142,10){\makebox(0,0){5}}
\end{picture} 
\begin{picture}(120,30) \thicklines
\put(61,18){\circle{8}}\put(61,-18){\circle{8}}\put(88,0){\circle{8}}\put(115,0){\circle{8}}
\put(142,0){\circle*{8}}\put(87,9){\makebox(0,0){3}} 
\put(85,2){\line(-3,2){20}}\put(85,-2){\line(-3,-2){20}}\put(85,0){\vertexbn}
\put(119,0){\line(1,0){19}}
\put(61,28){\makebox(0,0){2}} \put(61,-28){\makebox(0,0){1}}   \put(115,10){\makebox(0,0){4}} 
\put(142,10){\makebox(0,0){5}}
\end{picture} 

\vspace{1.5cm}

\begin{picture}(120,30) \thicklines
\put(61,18){\circle{8}}\put(61,-18){\circle{8}}\put(88,0){\circle{8}}\put(115,0){\circle{8}}
\put(142,0){\circle{8}}\put(87,9){\makebox(0,0){3}} 
\put(85,2){\line(-3,2){20}}\put(85,-2){\line(-3,-2){20}}\put(85,0){\vertexbn}
\put(119,0){\line(1,0){19}}
\put(61,28){\makebox(0,0){2}} \put(61,-28){\makebox(0,0){1}}   \put(115,10){\makebox(0,0){4}} 
\put(142,10){\makebox(0,0){5}}
 \qbezier(47,0)(47,10)(57,18)
        \qbezier(47,0)(47,-10)(57,-18)
\put(57,18){\vector( 1, 1){0}}
\put(57,-18){\vector( 1, -1){0}}
\end{picture}
\begin{picture}(120,30) \thicklines
\put(61,18){\circle{8}}\put(61,-18){\circle{8}}\put(88,0){\circle{8}}\put(115,0){\circle{8}}
\put(142,0){\circle*{8}}\put(87,9){\makebox(0,0){3}} 
\put(85,2){\line(-3,2){20}}\put(85,-2){\line(-3,-2){20}}\put(85,0){\vertexbn}
\put(119,0){\line(1,0){19}}
\put(61,28){\makebox(0,0){2}} \put(61,-28){\makebox(0,0){1}}   \put(115,10){\makebox(0,0){4}} 
\put(142,10){\makebox(0,0){5}}
 \qbezier(47,0)(47,10)(57,18)
        \qbezier(47,0)(47,-10)(57,-18)
\put(57,18){\vector( 1, 1){0}}
\put(57,-18){\vector( 1, -1){0}}
\end{picture}

\vspace{1cm}

\end{proof}
Note- There may be double painted Vogan diagrams non equivalent to single painted diagram. 
\noindent{\bf Example-5} The dynkin diagram of another rank 5 hyperbolic Kac-Moody algebra is given by

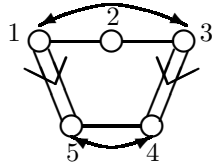
\begin{center}
\begin{picture}(80,30)\thicklines
\put(0,0){\circle{8}}\put(27,0){\circle{8}}\put(54,0){\circle{8}}
\put(12,-32){\circle{8}}\put(42,-32){\circle{8}}
\put(4,0){\line(1,0){19}}\put(31,0){\line(1,0){19}}
\put(16,-32){\line(1,0){22}}
\put(13.5,-28){\line(-2,5){10}}\put(8.5,-30){\line(-2,5){10.5}}
\put(40.5,-28){\line(2,5){10}}\put(45,-30){\line(2,5){10.5}}
\put(48,-16){\line(2,1){11.5}}\put(48,-16){\line(-2,5){4}}
\put(6,-16){\line(-2,1){11.5}}\put(6,-16){\line(2,5){4}}
\put(-12,0){1}\put(25,6){2}\put(60,0){3}\put(10,-45){5}\put(40,-45){4}
\qbezier(27,14)(14, 13)(0,6) \qbezier(27,14)(40, 13)(55, 6)
\put(0,6){\vector( -2, -1){0}}
\put(55,6){\vector( 2, -1){0}}
\qbezier(27,-40)(22,-40)(12,-36) \qbezier(27,-40)(32,-40)(42,-36)
\put(12,-36){\vector( -2,1){0}}
\put(42,-36){\vector( 2,1){0}}
\end{picture}\end{center}

\vspace{2cm}

\begin{prop} The equivalence classes of this rank 5  hyperbolic Kac-Moody algebra are given by
\begin{itemize}
 \item [(a)] $2\sim (1,2,3)$
 \item [(b)]$5\sim (4,5)\sim (2,3,4,5) \sim4\sim (4,5)$
 \item [(c)]$1\sim (1,2,5)\sim (2,3,5) \sim 3 \sim (2,3,4)\sim (1,2,4)$ 
\end{itemize}\end{prop}

\begin{proof}
 
\begin{enumerate}
\item [(a)] $\left(2\right)\sim\left(1,2,3\right)\mbox{ by 
}F\left(2\right)\sim\left(1,4,5\right)$
\item [(b)] $\left(5\right)\sim\left(4,5\right)\mbox{ by 
}F\left(5\right)\sim\left(2,3,4,5\right)\mbox{ by }F\left(3\right)$\\
and  $\left(4\right)\sim(4,5)\mbox{ by } F\langle 4\rangle$
\item [(c)]$\left(1\right)\sim (1,2,5)\mbox{ by } F\langle 1\rangle\sim (2,3,5)\mbox{ by } F\langle 2\rangle\sim (2,3,4) \mbox{ by } F\langle 4\rangle$\\
 \mbox{and} $3 \sim (2,3,4)\mbox{ by } F\langle 3\rangle\sim (1,2,4) \mbox{ by } F\langle 2\rangle$
\end{enumerate}
\end{proof}

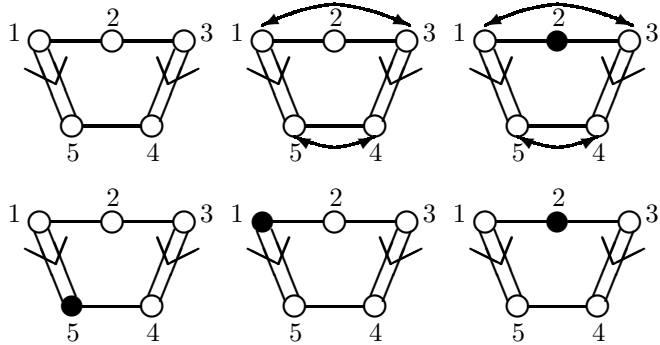
\begin{center}
\begin{picture}(80,30)\thicklines
\put(0,0){\circle{8}}\put(27,0){\circle{8}}\put(54,0){\circle{8}}
\put(12,-32){\circle{8}}\put(42,-32){\circle{8}}
\put(4,0){\line(1,0){19}}\put(31,0){\line(1,0){19}}
\put(16,-32){\line(1,0){22}}
\put(13.5,-28){\line(-2,5){10}}\put(8.5,-30){\line(-2,5){10.5}}
\put(40.5,-28){\line(2,5){10}}\put(45,-30){\line(2,5){10.5}}
\put(48,-16){\line(2,1){11.5}}\put(48,-16){\line(-2,5){4}}
\put(6,-16){\line(-2,1){11.5}}\put(6,-16){\line(2,5){4}}
\put(-12,0){1}\put(25,6){2}\put(60,0){3}\put(10,-45){5}\put(40,-45){4}
\end{picture}
\begin{picture}(80,30)\thicklines
\put(0,0){\circle{8}}\put(27,0){\circle{8}}\put(54,0){\circle{8}}
\put(12,-32){\circle{8}}\put(42,-32){\circle{8}}
\put(4,0){\line(1,0){19}}\put(31,0){\line(1,0){19}}
\put(16,-32){\line(1,0){22}}
\put(13.5,-28){\line(-2,5){10}}\put(8.5,-30){\line(-2,5){10.5}}
\put(40.5,-28){\line(2,5){10}}\put(45,-30){\line(2,5){10.5}}
\put(48,-16){\line(2,1){11.5}}\put(48,-16){\line(-2,5){4}}
\put(6,-16){\line(-2,1){11.5}}\put(6,-16){\line(2,5){4}}
\put(-12,0){1}\put(25,6){2}\put(60,0){3}\put(10,-45){5}\put(40,-45){4}
\qbezier(27,14)(14, 13)(0,6) \qbezier(27,14)(40, 13)(55, 6)
\put(0,6){\vector( -2, -1){0}}
\put(55,6){\vector( 2, -1){0}}
\qbezier(27,-40)(22,-40)(12,-36) \qbezier(27,-40)(32,-40)(42,-36)
\put(12,-36){\vector( -2,1){0}}
\put(42,-36){\vector( 2,1){0}}
\end{picture}
\begin{picture}(80,30)\thicklines
\put(0,0){\circle{8}}\put(27,0){\circle*{8}}\put(54,0){\circle{8}}
\put(12,-32){\circle{8}}\put(42,-32){\circle{8}}
\put(4,0){\line(1,0){19}}\put(31,0){\line(1,0){19}}
\put(16,-32){\line(1,0){22}}
\put(13.5,-28){\line(-2,5){10}}\put(8.5,-30){\line(-2,5){10.5}}
\put(40.5,-28){\line(2,5){10}}\put(45,-30){\line(2,5){10.5}}
\put(48,-16){\line(2,1){11.5}}\put(48,-16){\line(-2,5){4}}
\put(6,-16){\line(-2,1){11.5}}\put(6,-16){\line(2,5){4}}
\put(-12,0){1}\put(25,6){2}\put(60,0){3}\put(10,-45){5}\put(40,-45){4}
\qbezier(27,14)(14, 13)(0,6) \qbezier(27,14)(40, 13)(55, 6)
\put(0,6){\vector( -2, -1){0}}
\put(55,6){\vector( 2, -1){0}}
\qbezier(27,-40)(22,-40)(12,-36) \qbezier(27,-40)(32,-40)(42,-36)
\put(12,-36){\vector( -2,1){0}}
\put(42,-36){\vector( 2,1){0}}
\end{picture}
\end{center}

\vspace{1cm}

\begin{center}
\begin{picture}(80,30)\thicklines
\put(0,0){\circle{8}}\put(27,0){\circle{8}}\put(54,0){\circle{8}}
\put(12,-32){\circle*{8}}\put(42,-32){\circle{8}}
\put(4,0){\line(1,0){19}}\put(31,0){\line(1,0){19}}
\put(16,-32){\line(1,0){22}}
\put(13.5,-28){\line(-2,5){10}}\put(8.5,-30){\line(-2,5){10.5}}
\put(40.5,-28){\line(2,5){10}}\put(45,-30){\line(2,5){10.5}}
\put(48,-16){\line(2,1){11.5}}\put(48,-16){\line(-2,5){4}}
\put(6,-16){\line(-2,1){11.5}}\put(6,-16){\line(2,5){4}}
\put(-12,0){1}\put(25,6){2}\put(60,0){3}\put(10,-45){5}\put(40,-45){4}
\end{picture}
\begin{picture}(80,30)\thicklines
\put(0,0){\circle*{8}}\put(27,0){\circle{8}}\put(54,0){\circle{8}}
\put(12,-32){\circle{8}}\put(42,-32){\circle{8}}
\put(4,0){\line(1,0){19}}\put(31,0){\line(1,0){19}}
\put(16,-32){\line(1,0){22}}
\put(13.5,-28){\line(-2,5){10}}\put(8.5,-30){\line(-2,5){10.5}}
\put(40.5,-28){\line(2,5){10}}\put(45,-30){\line(2,5){10.5}}
\put(48,-16){\line(2,1){11.5}}\put(48,-16){\line(-2,5){4}}
\put(6,-16){\line(-2,1){11.5}}\put(6,-16){\line(2,5){4}}
\put(-12,0){1}\put(25,6){2}\put(60,0){3}\put(10,-45){5}\put(40,-45){4}
\end{picture}
\begin{picture}(80,30)\thicklines
\put(0,0){\circle{8}}\put(27,0){\circle*{8}}\put(54,0){\circle{8}}
\put(12,-32){\circle{8}}\put(42,-32){\circle{8}}
\put(4,0){\line(1,0){19}}\put(31,0){\line(1,0){19}}
\put(16,-32){\line(1,0){22}}
\put(13.5,-28){\line(-2,5){10}}\put(8.5,-30){\line(-2,5){10.5}}
\put(40.5,-28){\line(2,5){10}}\put(45,-30){\line(2,5){10.5}}
\put(48,-16){\line(2,1){11.5}}\put(48,-16){\line(-2,5){4}}
\put(6,-16){\line(-2,1){11.5}}\put(6,-16){\line(2,5){4}}
\put(-12,0){1}\put(25,6){2}\put(60,0){3}\put(10,-45){5}\put(40,-45){4}
\end{picture}\end{center}

\vspace{2cm}

Note- There may be double painted Vogan diagrams non equivalent to single painted diagram.
\noindent{\bf Example-6} The Dynkin diagram of another rank 5 hyperbolic Kac-Moody algebra is given by

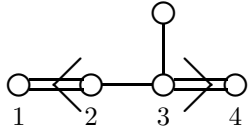
\begin{center}
\begin{picture}(120,30) \thicklines
 \put(0,0){\circle{8}}\put(27,0){\circle{8}}\put(54,0){\circle{8}}\put(81,0){\circle{8}}\put(54,27){\circle{8}}
\put(-2,0){\vertexcn} \put(31,0){\line(1,0){19}}\put(52,0){\vertexbn}\put(54,4){\line(0,1){19}}
\put(0,-12){\makebox(0,0){1}} \put(27,-12){\makebox(0,0){2}}   \put(54,-12){\makebox(0,0){3}}   \put(54,37){\makebox(0,0){5}} 
  \put(81,-12){\makebox(0,0){4}} 
\end{picture}
\end{center}

\vspace{0.5cm}

\begin{prop} The equivalence classes of this rank 5 
hyperbolic Kac-Moody algebra are given by
\begin{itemize}
 \item [(a)] 1
 \item [(b)] 4
 \item [(c)] 5
 \item [(d)] 3
 \item [(e)] 2
 \end{itemize}\end{prop}
\begin{proof}

\begin{enumerate}
\item[(a)] $1\sim 1$ by $F\langle 1\rangle$
\item[(b)] $ 4\sim 4$ by $F\langle 4\rangle$
\item[(c)] $5\sim (5,3)$ by $F\langle 5\rangle\sim(2,3,4)$ by $F\langle 3\rangle$
\item[(d)] $3\sim(2,3,4,5)$ by $F\langle 5\rangle$
\item[(e)] $2\sim (1,2,3)$ by $F\langle 2\rangle$
\end{enumerate}

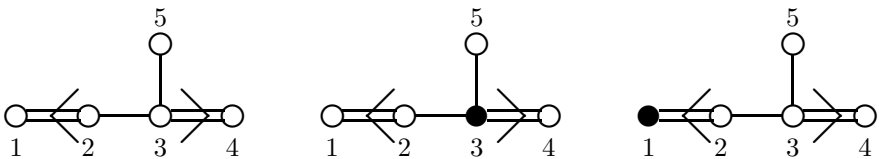
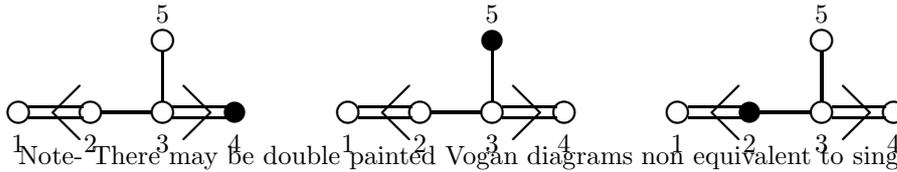
\vspace{1cm}

\begin{picture}(115,30) \thicklines
 \put(0,0){\circle{8}}\put(27,0){\circle{8}}\put(54,0){\circle{8}}\put(81,0){\circle{8}}\put(54,27){\circle{8}}
\put(-2,0){\vertexcn} \put(31,0){\line(1,0){19}}\put(52,0){\vertexbn}\put(54,4){\line(0,1){19}}
\put(0,-12){\makebox(0,0){1}} \put(27,-12){\makebox(0,0){2}}   \put(54,-12){\makebox(0,0){3}}   \put(54,37){\makebox(0,0){5}} 
  \put(81,-12){\makebox(0,0){4}} 
\end{picture}
\begin{picture}(115,30) \thicklines
 \put(0,0){\circle{8}}\put(27,0){\circle{8}}\put(54,0){\circle*{8}}\put(81,0){\circle{8}}\put(54,27){\circle{8}}
\put(-2,0){\vertexcn} \put(31,0){\line(1,0){19}}\put(52,0){\vertexbn}\put(54,4){\line(0,1){19}}
\put(0,-12){\makebox(0,0){1}} \put(27,-12){\makebox(0,0){2}}   \put(54,-12){\makebox(0,0){3}}   \put(54,37){\makebox(0,0){5}} 
  \put(81,-12){\makebox(0,0){4}} 
\end{picture}
\begin{picture}(115,30) \thicklines
 \put(0,0){\circle*{8}}\put(27,0){\circle{8}}\put(54,0){\circle{8}}\put(81,0){\circle{8}}\put(54,27){\circle{8}}
\put(-2,0){\vertexcn} \put(31,0){\line(1,0){19}}\put(52,0){\vertexbn}\put(54,4){\line(0,1){19}}
\put(0,-12){\makebox(0,0){1}} \put(27,-12){\makebox(0,0){2}}   \put(54,-12){\makebox(0,0){3}}   \put(54,37){\makebox(0,0){5}} 
  \put(81,-12){\makebox(0,0){4}} 
\end{picture}

\vspace{1.5cm}

\begin{picture}(120,30) \thicklines
 \put(0,0){\circle{8}}\put(27,0){\circle{8}}\put(54,0){\circle{8}}\put(81,0){\circle*{8}}\put(54,27){\circle{8}}
\put(-2,0){\vertexcn} \put(31,0){\line(1,0){19}}\put(52,0){\vertexbn}\put(54,4){\line(0,1){19}}
\put(0,-12){\makebox(0,0){1}} \put(27,-12){\makebox(0,0){2}}   \put(54,-12){\makebox(0,0){3}}   \put(54,37){\makebox(0,0){5}} 
  \put(81,-12){\makebox(0,0){4}} 
\end{picture}
\begin{picture}(120,30) \thicklines
 \put(0,0){\circle{8}}\put(27,0){\circle{8}}\put(54,0){\circle{8}}\put(81,0){\circle{8}}\put(54,27){\circle*{8}}
\put(-2,0){\vertexcn} \put(31,0){\line(1,0){19}}\put(52,0){\vertexbn}\put(54,4){\line(0,1){19}}
\put(0,-12){\makebox(0,0){1}} \put(27,-12){\makebox(0,0){2}}   \put(54,-12){\makebox(0,0){3}}   \put(54,37){\makebox(0,0){5}} 
  \put(81,-12){\makebox(0,0){4}} 
\end{picture}
\begin{picture}(120,30) \thicklines
 \put(0,0){\circle{8}}\put(27,0){\circle*{8}}\put(54,0){\circle{8}}\put(81,0){\circle{8}}\put(54,27){\circle{8}}
\put(-2,0){\vertexcn} \put(31,0){\line(1,0){19}}\put(52,0){\vertexbn}\put(54,4){\line(0,1){19}}
\put(0,-12){\makebox(0,0){1}} \put(27,-12){\makebox(0,0){2}}   \put(54,-12){\makebox(0,0){3}}   \put(54,37){\makebox(0,0){5}} 
  \put(81,-12){\makebox(0,0){4}} 
\end{picture}
\end{proof}
Note- There may be double painted Vogan diagrams non equivalent to single painted diagram.

\section{Vogan diagrams of some more hyperbolic Kac-Moody algebras and its relative diagrams}
 
 We will construct few relative Dynkin diagrams, we can find similarly  the relative diagrams of all Vogan diagrams.\\
 
  The Vogan diagrams of $GG_{3}$
  \begin{itemize}
 \item [(a)]$2\sim (1,2,3)$
 \item[(b)] $1\sim 3$ by symmetry
  \item [(c)]$(1,3)$ is unique
  \end{itemize}

     
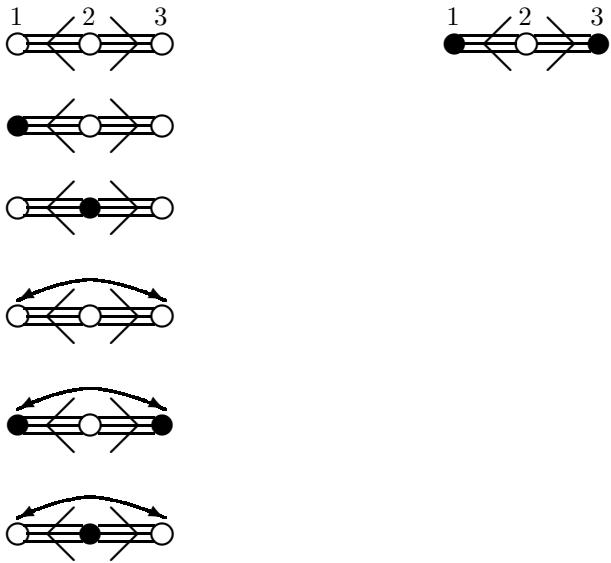
\begin{picture}(160,30) \thicklines
\put(7,0){\circle{8}}\put(34,0){\circle{8}}\put(61,0){\circle{8}}
\put(3,0){\vertexccn}\put(32,0){\vertexbbn}
\put(4,7){1}\put(31,7){2}\put(58,7){3}
  \end{picture}
      \begin{picture}(160,30) \thicklines
\put(7,0){\circle*{8}}\put(34,0){\circle{8}}\put(61,0){\circle*{8}}
\put(3,0){\vertexccn}\put(32,0){\vertexbbn}
\put(4,7){1}\put(31,7){2}\put(58,7){3}
  \end{picture}
 
       \begin{picture}(60,30) \thicklines
\put(7,0){\circle*{8}}\put(34,0){\circle{8}}\put(61,0){\circle{8}}
\put(3,0){\vertexccn}\put(32,0){\vertexbbn}
  \end{picture}
   
    \begin{picture}(60,30) \thicklines
\put(7,0){\circle{8}}\put(34,0){\circle*{8}}\put(61,0){\circle{8}}
\put(3,0){\vertexccn}\put(32,0){\vertexbbn}
  \end{picture}
  
 \begin{picture}(60,40) \thicklines
\put(7,0){\circle{8}}\put(34,0){\circle{8}}\put(61,0){\circle{8}}
\put(3,0){\vertexccn}\put(32,0){\vertexbbn}
\qbezier(34,14)(21, 13)(7, 6) \qbezier(34,14)(47, 13)(62, 6)
\put(7,6){\vector( -2, -1){0}}
\put(62,6){\vector( 2, -1){0}}
  \end{picture}
  
 \begin{picture}(60,40) \thicklines
\put(7,0){\circle*{8}}\put(34,0){\circle{8}}\put(61,0){\circle*{8}}
\put(3,0){\vertexccn}\put(32,0){\vertexbbn}
\qbezier(34,14)(21, 13)(7, 6) \qbezier(34,14)(47, 13)(62, 6)
\put(7,6){\vector( -2, -1){0}}
\put(62,6){\vector( 2, -1){0}}
  \end{picture}
  
 \begin{picture}(60,40) \thicklines
\put(7,0){\circle{8}}\put(34,0){\circle*{8}}\put(61,0){\circle{8}}
\put(3,0){\vertexccn}\put(32,0){\vertexbbn}
\qbezier(34,14)(21, 13)(7, 6) \qbezier(34,14)(47, 13)(62, 6)
\put(7,6){\vector( -2, -1){0}}
\put(62,6){\vector( 2, -1){0}}
  \end{picture}
  \vspace{1cm}
  
  From the above diagrams we get the relative Dynkin diagrams
  
    
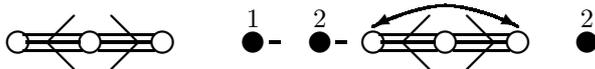
\begin{picture}(160,30) \thicklines
\put(7,0){\circle{8}}\put(34,0){\circle{8}}\put(61,0){\circle{8}}
\put(3,0){\vertexccn}\put(32,0){\vertexbbn}
\put(95,0){\circle*{8}}\put(120,0){\circle*{8}}
\put(140,0){\circle{8}}\put(167,0){\circle{8}}\put(194,0){\circle{8}}
\put(136,0){\vertexccn}\put(165,0){\vertexbbn}
\qbezier(167,14)(154, 13)(140, 6) \qbezier(167,14)(180, 13)(194, 6)
\put(140,6){\vector( -2, -1){0}}
\put(194,6){\vector( 2, -1){0}}
\put(220,0){\circle*{8}}
\put(126,0){\line(1,0){4}}\put(101,0){\line(1,0){4}}
     \put(95,9){\makebox(0,0){1}}  \put(120,9){\makebox(0,0){2}} \put(220,9){\makebox(0,0){2}}
  \end{picture}
    
  \vspace{1cm}
  
  The Vogan diagrams of $G'G_{3}$
  \begin{prop} The equivalence classes of this rank 3 hyperbolic Kac-Moody
algebra are given by
   \begin{itemize}
    \item [(a)]$1$
    \item [(b)]$2$
     \item [(c)]$3$
   \end{itemize}

  \end{prop}
  \begin{proof}
    $1\sim (1,2) \sim (1,2,3)\sim(1,3)$ by $F\langle 1,2,1\rangle$ and $2\sim(2,3)$ by $F\langle 2\rangle$.

  \end{proof}

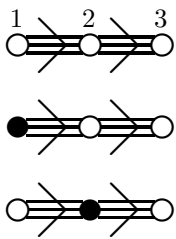
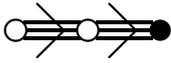
  \begin{picture}(60,30) \thicklines
  \put(4,7){1}\put(31,7){2}\put(58,7){3}
\put(7,0){\circle{8}}\put(34,0){\circle{8}}\put(61,0){\circle{8}}
\put(5,0){\vertexbbn}\put(32,0){\vertexbbn}
  \end{picture} 
  
    \begin{picture}(60,30) \thicklines
\put(7,0){\circle*{8}}\put(34,0){\circle{8}}\put(61,0){\circle{8}}
\put(5,0){\vertexbbn}\put(32,0){\vertexbbn}
  \end{picture}
  
     \begin{picture}(60,30) \thicklines
\put(7,0){\circle{8}}\put(34,0){\circle*{8}}\put(61,0){\circle{8}}
\put(5,0){\vertexbbn}\put(32,0){\vertexbbn}
  \end{picture}
  
     \begin{picture}(60,30) \thicklines
\put(7,0){\circle{8}}\put(34,0){\circle{8}}\put(61,0){\circle*{8}}
\put(5,0){\vertexbbn}\put(32,0){\vertexbbn}
  \end{picture}
  \vspace{1cm}
  
  Its relative Dynkin diagrams are
  
    
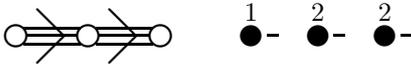
\begin{picture}(60,30) \thicklines
\put(7,0){\circle{8}}\put(34,0){\circle{8}}\put(61,0){\circle{8}}
\put(5,0){\vertexbbn}\put(32,0){\vertexbbn}
\put(95,0){\circle*{8}}\put(120,0){\circle*{8}}
\put(145,0){\circle*{8}}
\put(126,0){\line(1,0){4}}\put(101,0){\line(1,0){4}}\put(151,0){\line(1,0){4}}
     \put(95,9){\makebox(0,0){1}}  \put(120,9){\makebox(0,0){2}} \put(145,9){\makebox(0,0){2}}
  \end{picture} 
  
  \vspace{1cm}
  
   The Vogan diagrams of $G'G'_{3}$
   \begin{prop} The equivalence classes of this rank 3 hyperbolic Kac-Moody
algebra are given by
    \begin{itemize}
     \item[(a)]  $1\sim 3$
     \item [(b)] $2\sim2$ 
     \item [(c)] $\sim (1,3)$
    \end{itemize}

   \end{prop}
   \begin{proof}
    \item [(a)] By symmetry $1\sim 3$
    \item [(b)] $2\sim2$
    \item [(c)] $(1,2,3)\sim (1,3)$ by $F\langle 1\rangle$
   \end{proof}

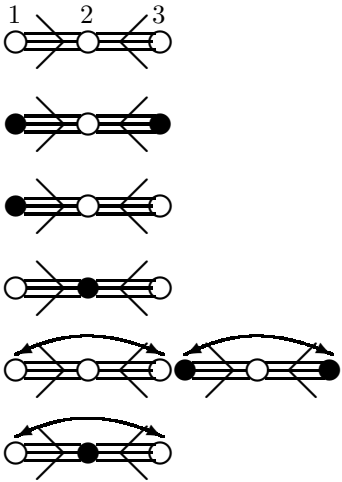
 \begin{picture}(60,30) \thicklines
 \put(4,7){1}\put(31,7){2}\put(58,7){3}
\put(7,0){\circle{8}}\put(34,0){\circle{8}}\put(61,0){\circle{8}}
\put(5,0){\vertexbbn}\put(31,0){\vertexccn}
  \end{picture} 
  
 \begin{picture}(60,30) \thicklines
\put(7,0){\circle*{8}}\put(34,0){\circle{8}}\put(61,0){\circle*{8}}
\put(5,0){\vertexbbn}\put(31,0){\vertexccn}
  \end{picture} 
  
   \begin{picture}(60,30) \thicklines
\put(7,0){\circle*{8}}\put(34,0){\circle{8}}\put(61,0){\circle{8}}
\put(5,0){\vertexbbn}\put(31,0){\vertexccn}
  \end{picture}
  
    \begin{picture}(60,30) \thicklines
\put(7,0){\circle{8}}\put(34,0){\circle*{8}}\put(61,0){\circle{8}}
\put(5,0){\vertexbbn}\put(31,0){\vertexccn}
  \end{picture} 
  
    \begin{picture}(60,30) \thicklines
\put(7,0){\circle{8}}\put(34,0){\circle{8}}\put(61,0){\circle{8}}
\put(5,0){\vertexbbn}\put(31,0){\vertexccn}
\qbezier(33,13)(21, 13)(7, 6) \qbezier(34,13)(47, 13)(62, 6)
\put(7,6){\vector( -2, -1){0}}
\put(62,6){\vector( 2, -1){0}}
  \end{picture}
      \begin{picture}(60,30) \thicklines
\put(7,0){\circle*{8}}\put(34,0){\circle{8}}\put(61,0){\circle*{8}}
\put(5,0){\vertexbbn}\put(31,0){\vertexccn}
\qbezier(33,13)(21, 13)(7, 6) \qbezier(34,13)(47, 13)(62, 6)
\put(7,6){\vector( -2, -1){0}}
\put(62,6){\vector( 2, -1){0}}
  \end{picture}
  
    \begin{picture}(60,30) \thicklines
\put(7,0){\circle{8}}\put(34,0){\circle*{8}}\put(61,0){\circle{8}}
\put(5,0){\vertexbbn}\put(31,0){\vertexccn}
\qbezier(33,13)(21, 13)(7, 6) \qbezier(34,13)(47, 13)(62, 6)
\put(7,6){\vector( -2, -1){0}}
\put(62,6){\vector( 2, -1){0}}
  \end{picture}
 \vspace{1cm} 
  
  We have the relative Dynkin diagrams for restricted roots
  
  
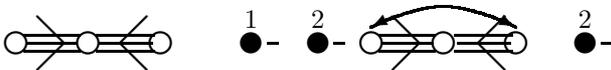
\begin{picture}(60,30) \thicklines
  \put(7,0){\circle{8}}\put(34,0){\circle{8}}\put(61,0){\circle{8}}
\put(5,0){\vertexbbn}\put(31,0){\vertexccn}
\put(95,0){\circle*{8}}\put(120,0){\circle*{8}}

\put(126,0){\line(1,0){4}}\put(101,0){\line(1,0){4}}\put(151,0){\line(1,0){4}}
     \put(95,9){\makebox(0,0){1}}  \put(120,9){\makebox(0,0){2}} \put(220,9){\makebox(0,0){2}}
\put(140,0){\circle{8}}\put(167,0){\circle{8}}\put(194,0){\circle{8}}
\put(138,0){\vertexbbn}\put(165,0){\vertexccn}
\qbezier(167,14)(154, 13)(140, 6) \qbezier(167,14)(180, 13)(194, 6)
\put(140,6){\vector( -2, -1){0}}
\put(194,6){\vector( 2, -1){0}}
\put(220,0){\circle*{8}}
\put(226,0){\line(1,0){4}}
  \end{picture}

  \vspace{1cm}
  
The Vogan diagrams of  $AC_{2}^{(1)}$

\begin{prop} The equivalence classes of this rank 3 hyperbolic Kac-Moody
algebra are given by
\begin{itemize}
\item  [(a)] $1$
 \item [(b)] $3$
 \item [(c)] $(1,2)$

\end{itemize}\end{prop}
\begin{proof}
 \begin{itemize}
 \item  [(a)] $1\sim (1,2,3)\sim 2$ by $F\langle 1,2\rangle$ 
 \item  [(b)] $3\sim3 $ by $F\langle 3\rangle$
 \item  [(c)] $(1,2)\sim(1,3)\sim (2,3)$ by $F\langle 1,1,2\rangle$
 \end{itemize}
 \end{proof}

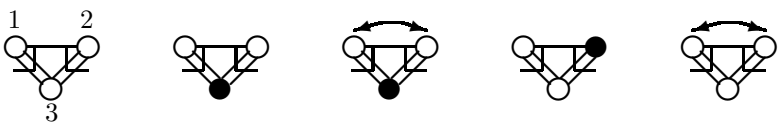
  \begin{picture}(60,30) \thicklines
  \put(4,7){1}\put(31,7){2}\put(18,-28){3}
\put(7,0){\circle{8}}\put(34,0){\circle{8}}\put(20,-16){\circle{8}}
 \put(11,0){\line(1,0){19}}\put(7,-4){\line(1,-1){10}}\put(10,-2){\line(1,-1){10}}
                           \put(31,-2){\line(-1,-1){10}}\put(34,-4){\line(-1,-1){10}}
                           
                           \put(6,-9){\line(1,0){8}} \put(14,-9){\line(0,1){8}}

                           \put(26,-9){\line(0,1){8}}\put(26,-9){\line(1,0){8}}
                      \end{picture}
    \begin{picture}(60,30) \thicklines
\put(7,0){\circle{8}}\put(34,0){\circle{8}}\put(20,-16){\circle*{8}}
 \put(11,0){\line(1,0){19}}\put(7,-4){\line(1,-1){10}}\put(10,-2){\line(1,-1){10}}
                           \put(31,-2){\line(-1,-1){10}}\put(34,-4){\line(-1,-1){10}}
                           
                           \put(6,-9){\line(1,0){8}} \put(14,-9){\line(0,1){8}}
                          
                                   \put(26,-9){\line(0,1){8}}\put(26,-9){\line(1,0){8}}
                    \end{picture}
    \begin{picture}(60,30) \thicklines
\put(7,0){\circle{8}}\put(34,0){\circle{8}}\put(20,-16){\circle*{8}}
 \put(11,0){\line(1,0){19}}\put(7,-4){\line(1,-1){10}}\put(10,-2){\line(1,-1){10}}
                           \put(31,-2){\line(-1,-1){10}}\put(34,-4){\line(-1,-1){10}}
                           
                           \put(6,-9){\line(1,0){8}} \put(14,-9){\line(0,1){8}}

                           \put(26,-9){\line(0,1){8}}\put(26,-9){\line(1,0){8}}
                    \qbezier(21,9)(23,10)(34,6) \qbezier(21,9)(19,10)(7,6)  
                    \put(7,6){\vector( -3, -1){0}}
\put(34,6){\vector( 3, -1){0}}
  \end{picture}
    \begin{picture}(60,30) \thicklines
\put(7,0){\circle{8}}\put(34,0){\circle*{8}}\put(20,-16){\circle{8}}
 \put(11,0){\line(1,0){19}}\put(7,-4){\line(1,-1){10}}\put(10,-2){\line(1,-1){10}}
                           \put(31,-2){\line(-1,-1){10}}\put(34,-4){\line(-1,-1){10}}
                           
                           \put(6,-9){\line(1,0){8}} \put(14,-9){\line(0,1){8}}

                           \put(26,-9){\line(0,1){8}}\put(26,-9){\line(1,0){8}}
                    \end{picture}
    \begin{picture}(60,30) \thicklines
\put(7,0){\circle{8}}\put(34,0){\circle{8}}\put(20,-16){\circle{8}}
 \put(11,0){\line(1,0){19}}\put(7,-4){\line(1,-1){10}}\put(10,-2){\line(1,-1){10}}
                           \put(31,-2){\line(-1,-1){10}}\put(34,-4){\line(-1,-1){10}}
                           
                           \put(6,-9){\line(1,0){8}} \put(14,-9){\line(0,1){8}}

                           \put(26,-9){\line(0,1){8}}\put(26,-9){\line(1,0){8}}
                    \qbezier(21,9)(23,10)(34,6) \qbezier(21,9)(19,10)(7,6)  
                    \put(7,6){\vector( -3, -1){0}}
\put(34,6){\vector( 3, -1){0}}
  \end{picture}
  \vspace{1cm}
  
The Vogan diagrams of  $AD_{3}^{(2)}$ \\
\begin{prop} The equivalence classes of this rank 3 hyperbolic Kac-Moody
algebra are given by
\begin{itemize}
\item  [(a)] $1$

\end{itemize}\end{prop}
\begin{proof}
 \begin{itemize}
 \item  [(a)]$1\sim (1,2)\sim 2$ by $F\langle 1,2\rangle$ also by symmetry
 \item  [(b)]$3\sim(1,2,3)\sim(2,3)\sim(1,2)\sim(1,3)$
 \end{itemize}
 \end{proof} 

      
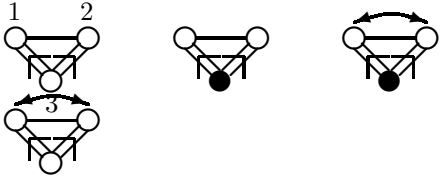
\begin{picture}(60,30) \thicklines
      \put(4,7){1}\put(31,7){2}\put(18,-28){3}
\put(7,0){\circle{8}}\put(34,0){\circle{8}}\put(20,-16){\circle{8}}
 \put(11,0){\line(1,0){19}}\put(7,-4){\line(1,-1){10}}\put(10,-2){\line(1,-1){10}}
                           \put(31,-2){\line(-1,-1){10}}\put(34,-4){\line(-1,-1){10}}
                           
                           \put(12,-7){\line(1,0){8}} \put(12,-7){\line(0,-1){8}}

                           \put(29,-7){\line(-1,0){8}}\put(29,-7){\line(0,-1){8}}
                     \end{picture}
        \begin{picture}(60,30) \thicklines
\put(7,0){\circle{8}}\put(34,0){\circle{8}}\put(20,-16){\circle*{8}}
 \put(11,0){\line(1,0){19}}\put(7,-4){\line(1,-1){10}}\put(10,-2){\line(1,-1){10}}
                           \put(31,-2){\line(-1,-1){10}}\put(34,-4){\line(-1,-1){10}}
                           
                           \put(12,-7){\line(1,0){8}} \put(12,-7){\line(0,-1){8}}

                           \put(29,-7){\line(-1,0){8}}\put(29,-7){\line(0,-1){8}}
                    \end{picture}
     \begin{picture}(60,30) \thicklines
\put(7,0){\circle{8}}\put(34,0){\circle{8}}\put(20,-16){\circle*{8}}
 \put(11,0){\line(1,0){19}}\put(7,-4){\line(1,-1){10}}\put(10,-2){\line(1,-1){10}}
                           \put(31,-2){\line(-1,-1){10}}\put(34,-4){\line(-1,-1){10}}
                           
                           \put(12,-7){\line(1,0){8}} \put(12,-7){\line(0,-1){8}}

                           \put(29,-7){\line(-1,0){8}}\put(29,-7){\line(0,-1){8}}
                    \qbezier(21,9)(23,10)(34,6) \qbezier(21,9)(19,10)(7,6)  
                    \put(7,6){\vector( -3, -1){0}}
\put(34,6){\vector( 3, -1){0}}
  \end{picture}
   
     \begin{picture}(60,30) \thicklines
\put(7,0){\circle{8}}\put(34,0){\circle{8}}\put(20,-16){\circle{8}}
 \put(11,0){\line(1,0){19}}\put(7,-4){\line(1,-1){10}}\put(10,-2){\line(1,-1){10}}
                  
                  \put(31,-2){\line(-1,-1){10}}\put(34,-4){\line(-1,-1){10}}
                           
                           \put(12,-7){\line(1,0){8}} \put(12,-7){\line(0,-1){8}}

                           \put(29,-7){\line(-1,0){8}}\put(29,-7){\line(0,-1){8}}
                    \qbezier(21,9)(23,10)(34,6) \qbezier(21,9)(19,10)(7,6)  
                    \put(7,6){\vector( -3, -1){0}}
\put(34,6){\vector( 3, -1){0}}
  \end{picture}
  \vspace{1cm}
    
  The Vogan diagrams of  $AGG_{3}$ \\
  
\begin{prop} The equivalence classes of this rank 3 hyperbolic Kac-Moody
algebra are given by
\begin{itemize}
\item  [(a)] $1$

\end{itemize}\end{prop}
\begin{proof}
 \begin{itemize}
 \item  [(a)]$1\sim (1,2)\sim 2$ by $F\langle 1,2 \rangle$ also by symmetry
 \item  [(b)] $3\sim(1,2,3)\sim(2,3)\sim(1,2)\sim(1,3)$
 \end{itemize}
 \end{proof}

        
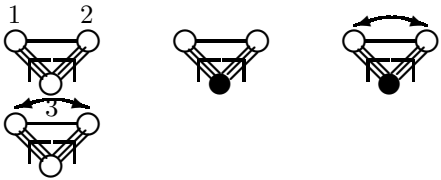
\begin{picture}(60,30) \thicklines
        \put(4,7){1}\put(31,7){2}\put(18,-28){3}
\put(7,0){\circle{8}}\put(34,0){\circle{8}}\put(20,-16){\circle{8}}
 \put(11,0){\line(1,0){19}}\put(7,-4){\line(1,-1){10}}\put(8.5,-3){\line(1,-1){10}}\put(10,-2){\line(1,-1){10}}
                           \put(31,-2){\line(-1,-1){10}}\put(33,-2.5){\line(-1,-1){10}}\put(34,-4){\line(-1,-1){10}}

                           \put(12,-7){\line(1,0){8}} \put(12,-7){\line(0,-1){8}}

                           \put(29,-7){\line(-1,0){8}}\put(29,-7){\line(0,-1){8}}
                  \end{picture}
         \begin{picture}(60,30) \thicklines
\put(7,0){\circle{8}}\put(34,0){\circle{8}}\put(20,-16){\circle*{8}}
 \put(11,0){\line(1,0){19}}\put(7,-4){\line(1,-1){10}}\put(8.5,-3){\line(1,-1){10}}\put(10,-2){\line(1,-1){10}}
                           \put(31,-2){\line(-1,-1){10}}\put(33,-2.5){\line(-1,-1){10}}\put(34,-4){\line(-1,-1){10}}

                           \put(12,-7){\line(1,0){8}} \put(12,-7){\line(0,-1){8}}

                           \put(29,-7){\line(-1,0){8}}\put(29,-7){\line(0,-1){8}}
                  \end{picture}
        \begin{picture}(60,30) \thicklines
\put(7,0){\circle{8}}\put(34,0){\circle{8}}\put(20,-16){\circle*{8}}
 \put(11,0){\line(1,0){19}}\put(7,-4){\line(1,-1){10}}\put(8.5,-3){\line(1,-1){10}}\put(10,-2){\line(1,-1){10}}
                           \put(31,-2){\line(-1,-1){10}}\put(33,-2.5){\line(-1,-1){10}}\put(34,-4){\line(-1,-1){10}}

                           \put(12,-7){\line(1,0){8}} \put(12,-7){\line(0,-1){8}}

                           \put(29,-7){\line(-1,0){8}}\put(29,-7){\line(0,-1){8}}
                    \qbezier(21,9)(23,10)(34,6) \qbezier(21,9)(19,10)(7,6)  
                    \put(7,6){\vector( -3, -1){0}}
\put(34,6){\vector( 3, -1){0}}
  \end{picture}
   
        \begin{picture}(60,30) \thicklines
\put(7,0){\circle{8}}\put(34,0){\circle{8}}\put(20,-16){\circle{8}}
 \put(11,0){\line(1,0){19}}\put(7,-4){\line(1,-1){10}}\put(8.5,-3){\line(1,-1){10}}\put(10,-2){\line(1,-1){10}}
                           \put(31,-2){\line(-1,-1){10}}\put(33,-2.5){\line(-1,-1){10}}\put(34,-4){\line(-1,-1){10}}

                           \put(12,-7){\line(1,0){8}} \put(12,-7){\line(0,-1){8}}

                           \put(29,-7){\line(-1,0){8}}\put(29,-7){\line(0,-1){8}}
                    \qbezier(21,9)(23,10)(34,6) \qbezier(21,9)(19,10)(7,6)  
                    \put(7,6){\vector( -3, -1){0}}
\put(34,6){\vector( 3, -1){0}}
  \end{picture}
  \vspace{1cm}
  
The Vogan diagrams of  $AG'G'_{3}$ \\
\begin{prop}The equivalence classes of this rank 3 hyperbolic Kac-Moody
algebra are given by
\begin{itemize}
\item  [(a)] $1$
 \item [(b)] $3$
 \item [(c)] $(1,2)$

\end{itemize}\end{prop}
\begin{proof}
 \begin{itemize}
 \item  [(a)] $1\sim (1,2,3)\sim 2$ by $F\langle 1,2\rangle$ 
 \item  [(b)] $3\sim3 $ by $F\langle 3\rangle$
 \item  [(c)] $(1,2)\sim(1,3)\sim (2,3)$ by $F\langle 1,1,2\rangle$
 \end{itemize}
 \end{proof}

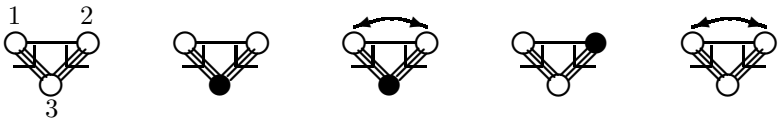
    \begin{picture}(60,30) \thicklines
    \put(4,7){1}\put(31,7){2}\put(18,-28){3}
\put(7,0){\circle{8}}\put(34,0){\circle{8}}\put(20,-16){\circle{8}}
 \put(11,0){\line(1,0){19}}\put(7,-4){\line(1,-1){10}}\put(8.5,-3){\line(1,-1){10}}\put(10,-2){\line(1,-1){10}}
                           \put(31,-2){\line(-1,-1){10}}\put(33,-2.5){\line(-1,-1){10}}\put(34,-4){\line(-1,-1){10}}
                           
                           \put(6,-9){\line(1,0){8}} \put(14,-9){\line(0,1){8}}

                           \put(26,-9){\line(0,1){8}}\put(26,-9){\line(1,0){8}}
               \end{picture}
      \begin{picture}(60,30) \thicklines
\put(7,0){\circle{8}}\put(34,0){\circle{8}}\put(20,-16){\circle*{8}}
 \put(11,0){\line(1,0){19}}\put(7,-4){\line(1,-1){10}}\put(8.5,-3){\line(1,-1){10}}\put(10,-2){\line(1,-1){10}}
                           \put(31,-2){\line(-1,-1){10}}\put(33,-2.5){\line(-1,-1){10}}\put(34,-4){\line(-1,-1){10}}
                           
                           \put(6,-9){\line(1,0){8}} \put(14,-9){\line(0,1){8}}

                           \put(26,-9){\line(0,1){8}}\put(26,-9){\line(1,0){8}}
             \end{picture}
      \begin{picture}(60,30) \thicklines
\put(7,0){\circle{8}}\put(34,0){\circle{8}}\put(20,-16){\circle*{8}}
 \put(11,0){\line(1,0){19}}\put(7,-4){\line(1,-1){10}}\put(8.5,-3){\line(1,-1){10}}\put(10,-2){\line(1,-1){10}}
                           \put(31,-2){\line(-1,-1){10}}\put(33,-2.5){\line(-1,-1){10}}\put(34,-4){\line(-1,-1){10}}
                           
                           \put(6,-9){\line(1,0){8}} \put(14,-9){\line(0,1){8}}

                           \put(26,-9){\line(0,1){8}}\put(26,-9){\line(1,0){8}}
                    \qbezier(21,9)(23,10)(34,6) \qbezier(21,9)(19,10)(7,6)  
                    \put(7,6){\vector( -3, -1){0}}
\put(34,6){\vector( 3, -1){0}}
  \end{picture}
      \begin{picture}(60,30) \thicklines
\put(7,0){\circle{8}}\put(34,0){\circle*{8}}\put(20,-16){\circle{8}}
 \put(11,0){\line(1,0){19}}\put(7,-4){\line(1,-1){10}}\put(8.5,-3){\line(1,-1){10}}\put(10,-2){\line(1,-1){10}}
                           \put(31,-2){\line(-1,-1){10}}\put(33,-2.5){\line(-1,-1){10}}\put(34,-4){\line(-1,-1){10}}
                           
                           \put(6,-9){\line(1,0){8}} \put(14,-9){\line(0,1){8}}

                           \put(26,-9){\line(0,1){8}}\put(26,-9){\line(1,0){8}}
                     \end{picture}
        \begin{picture}(60,30) \thicklines
\put(7,0){\circle{8}}\put(34,0){\circle{8}}\put(20,-16){\circle{8}}
 \put(11,0){\line(1,0){19}}\put(7,-4){\line(1,-1){10}}\put(8.5,-3){\line(1,-1){10}}\put(10,-2){\line(1,-1){10}}
                           \put(31,-2){\line(-1,-1){10}}\put(33,-2.5){\line(-1,-1){10}}\put(34,-4){\line(-1,-1){10}}
                           
                           \put(6,-9){\line(1,0){8}} \put(14,-9){\line(0,1){8}}

                           \put(26,-9){\line(0,1){8}}\put(26,-9){\line(1,0){8}}
                    \qbezier(21,9)(23,10)(34,6) \qbezier(21,9)(19,10)(7,6)  
                    \put(7,6){\vector( -3, -1){0}}
\put(34,6){\vector( 3, -1){0}}
  \end{picture}
  \vspace{1cm}
  
 The Vogan diagram of $AC_{3,4}^{(1)}$\\
  \begin{prop} The equivalence classes of this rank 4 hyperbolic Kac-Moody
algebra are given by
\begin{itemize}
\item  [(a)] $1$
 \item [(b)] $2$

\end{itemize}\end{prop}
\begin{proof}
 \begin{itemize}
 \item  [(a)] By symmetry $1\sim4$
 \item  [(b)] By symmetry $2\sim3$ 
 \end{itemize}
 \end{proof}      
       
 \begin{picture}(60,30) \thicklines
 \put(7,0){\circle{8}}\put(34,0){\circle{8}}
 \put(4,0){\vertexbn}
 \put(7,-22){\circle{8}}\put(34,-22){\circle{8}}
 \put(4,-22){\vertexbn}
 \put(7,-4){\line(0,-1){14}}\put(34,-4){\line(0,-1){14}}
 \end{picture}
  \begin{picture}(60,30) \thicklines
 \put(7,0){\circle{8}}\put(34,0){\circle{8}}
 \put(4,0){\vertexbn}
 \put(7,-22){\circle{8}}\put(34,-22){\circle{8}}
 \put(4,-22){\vertexbn}
 \put(7,-4){\line(0,-1){14}}\put(34,-4){\line(0,-1){14}}
 \end{picture}
  \begin{picture}(60,30) \thicklines
 \put(7,0){\circle*{8}}\put(34,0){\circle{8}}
 \put(4,0){\vertexbn}
 \put(7,-22){\circle{8}}\put(34,-22){\circle{8}}
 \put(4,-22){\vertexbn}
 \put(7,-4){\line(0,-1){14}}\put(34,-4){\line(0,-1){14}}
 \end{picture}
  \begin{picture}(60,30) \thicklines
 \put(7,0){\circle{8}}\put(34,0){\circle*{8}}
 \put(4,0){\vertexbn}
 \put(7,-22){\circle{8}}\put(34,-22){\circle{8}}
 \put(4,-22){\vertexbn}
 \put(7,-4){\line(0,-1){14}}\put(34,-4){\line(0,-1){14}}
 \end{picture}
  \begin{picture}(60,30) \thicklines
 \put(7,0){\circle{8}}\put(34,0){\circle{8}}
 \put(4,0){\vertexbn}
 \put(7,-22){\circle{8}}\put(34,-22){\circle{8}}
 \put(4,-22){\vertexbn}
 \put(7,-4){\line(0,-1){14}}\put(34,-4){\line(0,-1){14}}
 
 \qbezier(-4,-10)(-3,-3)(1,0)    \qbezier(-4,-10)(-3,-17)(1,-22)
 \put(1,0){\vector( 1, 1){2}} \put(1,-22){\vector( 1, -1){2}}

 \qbezier(45,-10)(44,-3)(40,0)    \qbezier(45,-10)(44,-17)(40,-22)
  \put(40,0){\vector( -1, 1){2}} \put(40,-22){\vector( -1, -1){2}}
 \end{picture}
  
  \vspace{1.5cm}

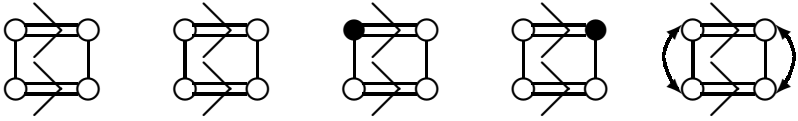
      Vogan diagrams of $HG_{2}^{(1)}$\\
\begin{prop} The equivalence classes of this rank 4 hyperbolic Kac-Moody
algebra are given by
\begin{itemize}
\item  [(a)] $1$
 \item [(b)] $2\sim 4$
 \item [(c)] $3$
 \item [(d)] $4$
\end{itemize}\end{prop}
\begin{proof}
 \begin{itemize}
 \item  [(a)] $1\sim 1$ by $F\langle 1\rangle$
 \item  [(b)]  $2\sim(1,2,3)\sim(1,3,4)\sim4$ by $F\langle 2,3,4\rangle$
 \item  [(c)] $3\sim(2,3,4)\sim(2,4)$ by $F\langle 3,4\rangle$
 \item  [(d)] $4\sim(3,4)\sim(2,3)\sim(1,2)$
 \end{itemize}
 \end{proof}

     \begin{picture}(60,30) \thicklines
    \put(7,0){\circle{8}}\put(34,0){\circle{8}}\put(61,0){\circle{8}}\put(88,0){\circle{8}}
    \put(4,0){\vertexccn}\put(38,0){\line(1,0){19}}\put(65,0){\line(1,0){19}}
    \put(7,10){\makebox(0,0){$1$}}\put(34,10){\makebox(0,0){$2$}}\put(61,10){\makebox(0,0){$3$}}
\put(88,10){\makebox(0,0){$4$}}
     \end{picture}
     
     \begin{picture}(60,30) \thicklines
    \put(7,0){\circle{8}}\put(34,0){\circle{8}}\put(61,0){\circle*{8}}\put(88,0){\circle{8}}
    \put(4,0){\vertexccn}\put(38,0){\line(1,0){19}}\put(65,0){\line(1,0){19}}
    \put(7,10){\makebox(0,0){$1$}}\put(34,10){\makebox(0,0){$2$}}\put(61,10){\makebox(0,0){$3$}}
\put(88,10){\makebox(0,0){$4$}}
     \end{picture}
     
      \begin{picture}(60,30) \thicklines
    \put(7,0){\circle*{8}}\put(34,0){\circle{8}}\put(61,0){\circle{8}}\put(88,0){\circle{8}}
    \put(4,0){\vertexccn}\put(38,0){\line(1,0){19}}\put(65,0){\line(1,0){19}}
    \put(7,10){\makebox(0,0){$1$}}\put(34,10){\makebox(0,0){$2$}}\put(61,10){\makebox(0,0){$3$}}
\put(88,10){\makebox(0,0){$4$}}
      \end{picture}
     
     \begin{picture}(60,30) \thicklines
    \put(7,0){\circle{8}}\put(34,0){\circle*{8}}\put(61,0){\circle{8}}\put(88,0){\circle{8}}
    \put(4,0){\vertexccn}\put(38,0){\line(1,0){19}}\put(65,0){\line(1,0){19}}
    \put(7,10){\makebox(0,0){$1$}}\put(34,10){\makebox(0,0){$2$}}\put(61,10){\makebox(0,0){$3$}}
\put(88,10){\makebox(0,0){$4$}}
     \end{picture}
     
      \begin{picture}(60,30) \thicklines
    \put(7,0){\circle{8}}\put(34,0){\circle{8}}\put(61,0){\circle{8}}\put(88,0){\circle*{8}}
    \put(4,0){\vertexccn}\put(38,0){\line(1,0){19}}\put(65,0){\line(1,0){19}}
    \put(7,10){\makebox(0,0){$1$}}\put(34,10){\makebox(0,0){$2$}}\put(61,10){\makebox(0,0){$3$}}
\put(88,10){\makebox(0,0){$4$}}
      \end{picture}
     
     \vspace{1cm}
     
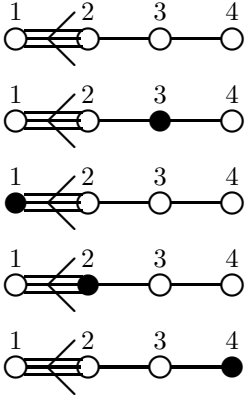
The Vogan diagram of $HF_{4}^{(1)}$\\
\begin{prop} The equivalence classes of this rank 6 hyperbolic Kac-Moody
algebra are given by
\begin{itemize}
\item  [(a)] $1$
 \item [(b)] $2$
 \item [(c)] $3$
 \item [(d)] $4$
 \item [(e)] $5$
\end{itemize}\end{prop}
\begin{proof}
 \begin{itemize}
 \item  [(a)]  $1\sim(1,2)\sim(2,3)\sim(3,4)\sim(4,5)\sim(4,5,6)\sim(4,6)$ by $F\langle 1,2,3,4,5,6\rangle$
 \item  [(b)]  $2\sim(1,2,3)\sim(1,3,4)\sim(1,4,5)\sim(1,4,5,6)\sim(1,4,6)$ by $F\langle2,3,4,5,6\rangle$
 \item  [(c)]  $3\sim(2,3,4)\sim(2,4,5)\sim(2,4,5,6)\sim(2,4,5,6)\sim(2,4,6)$ by $F\langle 3,4,5,6\rangle$
 \item  [(d)]  $4\sim(3,4,5)\sim(3,4,5,6)\sim(3,4,6)$ by $F\langle 4,5,6\rangle$
 \item  [(e)]  $5\sim(5,6)\sim 6$ by $F\langle 5,6\rangle$
 \end{itemize}
 \end{proof}

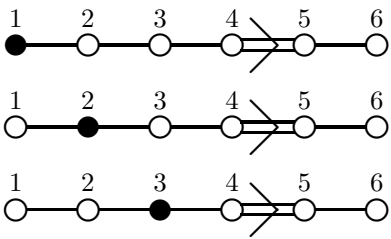
\begin{picture}(60,30) \thicklines
\put(7,0){\circle*{8}}\put(34,0){\circle{8}}\put(61,0){\circle{8}}\put(88,0){\circle{8}}\put(115,0){\circle{8}}\put(142,0){\circle{8}}
\put(11,0){\line(1,0){19}}\put(38,0){\line(1,0){19}}\put(65,0){\line(1,0){19}}\put(85,0){\vertexbn}\put(119,0){\line(1,0){19}}
\put(7,10){\makebox(0,0){$1$}}\put(34,10){\makebox(0,0){$2$}}\put(61,10){\makebox(0,0){$3$}}
\put(88,10){\makebox(0,0){$4$}}\put(115,10){\makebox(0,0){$5$}}\put(142,10){\makebox(0,0){$6$}}
\end{picture} 

  \begin{picture}(60,30) \thicklines
\put(7,0){\circle{8}}\put(34,0){\circle*{8}}\put(61,0){\circle{8}}\put(88,0){\circle{8}}\put(115,0){\circle{8}}\put(142,0){\circle{8}}
\put(11,0){\line(1,0){19}}\put(38,0){\line(1,0){19}}\put(65,0){\line(1,0){19}}\put(85,0){\vertexbn}\put(119,0){\line(1,0){19}}
\put(7,10){\makebox(0,0){$1$}}\put(34,10){\makebox(0,0){$2$}}\put(61,10){\makebox(0,0){$3$}}
\put(88,10){\makebox(0,0){$4$}}\put(115,10){\makebox(0,0){$5$}}\put(142,10){\makebox(0,0){$6$}}
\end{picture} 

\begin{picture}(60,30) \thicklines
\put(7,0){\circle{8}}\put(34,0){\circle{8}}\put(61,0){\circle*{8}}\put(88,0){\circle{8}}\put(115,0){\circle{8}}\put(142,0){\circle{8}}
\put(11,0){\line(1,0){19}}\put(38,0){\line(1,0){19}}\put(65,0){\line(1,0){19}}\put(85,0){\vertexbn}\put(119,0){\line(1,0){19}}
\put(7,10){\makebox(0,0){$1$}}\put(34,10){\makebox(0,0){$2$}}\put(61,10){\makebox(0,0){$3$}}
\put(88,10){\makebox(0,0){$4$}}\put(115,10){\makebox(0,0){$5$}}\put(142,10){\makebox(0,0){$6$}}
\end{picture}

\begin{picture}(60,30) \thicklines
\put(7,0){\circle{8}}\put(34,0){\circle{8}}\put(61,0){\circle{8}}\put(88,0){\circle*{8}}\put(115,0){\circle{8}}\put(142,0){\circle{8}}
\put(11,0){\line(1,0){19}}\put(38,0){\line(1,0){19}}\put(65,0){\line(1,0){19}}\put(85,0){\vertexbn}\put(119,0){\line(1,0){19}}
\put(7,10){\makebox(0,0){$1$}}\put(34,10){\makebox(0,0){$2$}}\put(61,10){\makebox(0,0){$3$}}
\put(88,10){\makebox(0,0){$4$}}\put(115,10){\makebox(0,0){$5$}}\put(142,10){\makebox(0,0){$6$}}
\end{picture} 

\begin{picture}(60,30) \thicklines
\put(7,0){\circle{8}}\put(34,0){\circle{8}}\put(61,0){\circle{8}}\put(88,0){\circle{8}}\put(115,0){\circle*{8}}\put(142,0){\circle{8}}
\put(11,0){\line(1,0){19}}\put(38,0){\line(1,0){19}}\put(65,0){\line(1,0){19}}\put(85,0){\vertexbn}\put(119,0){\line(1,0){19}}
\put(7,10){\makebox(0,0){$1$}}\put(34,10){\makebox(0,0){$2$}}\put(61,10){\makebox(0,0){$3$}}
\put(88,10){\makebox(0,0){$4$}}\put(115,10){\makebox(0,0){$5$}}\put(142,10){\makebox(0,0){$6$}}
\end{picture}

\begin{picture}(60,30) \thicklines
\put(7,0){\circle{8}}\put(34,0){\circle{8}}\put(61,0){\circle{8}}\put(88,0){\circle{8}}\put(115,0){\circle{8}}\put(142,0){\circle*{8}}
\put(11,0){\line(1,0){19}}\put(38,0){\line(1,0){19}}\put(65,0){\line(1,0){19}}\put(85,0){\vertexbn}\put(119,0){\line(1,0){19}}
\put(7,10){\makebox(0,0){$1$}}\put(34,10){\makebox(0,0){$2$}}\put(61,10){\makebox(0,0){$3$}}
\put(88,10){\makebox(0,0){$4$}}\put(115,10){\makebox(0,0){$5$}}\put(142,10){\makebox(0,0){$6$}}
\end{picture}

\vspace{0.5cm}

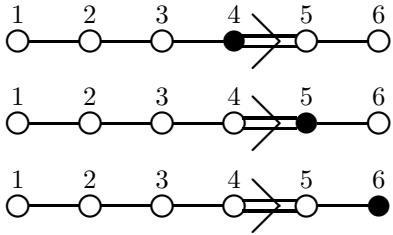
The Vogan diagram of $HA_{2}^{(2)}$ 
\begin{prop}
 The equivalence clasess of Vogan diagrams on rank 3 hyperbolic Kac-Moody algebra are given by
 \begin{itemize}
  \item[(a)] $1\sim2$
  \item [(b)]$3$
 \end{itemize}
 \end{prop}
 \begin{proof}
  \begin{itemize}
   \item [(a)]$1\sim(1,2)\sim (2,3)\sim2$ by $F\langle 1,2\rangle$
   \item[(b)] $3\sim3$ by $F\langle 3\rangle$
    \item [(c)]$1\sim(1,3)\sim (1,2,3)\sim2$ by $F\langle 1,2\rangle$
  \end{itemize}

 \end{proof}

\vspace{0.5cm}
 \begin{picture}(160,30) \thicklines
\put(7,0){\circle{8}}\put(34,0){\circle{8}}\put(61,0){\circle{8}}
\put(11,0){\line(1,0){19}}\put(32,0){\vertexbbbn}
\put(7,10){\makebox(0,0){$1$}}\put(34,10){\makebox(0,0){$2$}}\put(61,10){\makebox(0,0){$3$}}
  \end{picture}
  \begin{picture}(160,30) \thicklines
\put(7,0){\circle*{8}}\put(34,0){\circle{8}}\put(61,0){\circle{8}}
\put(11,0){\line(1,0){19}}\put(32,0){\vertexbbbn}
\put(7,10){\makebox(0,0){$1$}}\put(34,10){\makebox(0,0){$2$}}\put(61,10){\makebox(0,0){$3$}}
  \end{picture}
  
\vspace{0.5cm}
  \begin{picture}(160,30) \thicklines
\put(7,0){\circle{8}}\put(34,0){\circle{8}}\put(61,0){\circle*{8}}
\put(11,0){\line(1,0){19}}\put(32,0){\vertexbbbn}
\put(7,10){\makebox(0,0){$1$}}\put(34,10){\makebox(0,0){$2$}}\put(61,10){\makebox(0,0){$3$}}
  \end{picture}
  \begin{picture}(160,30) \thicklines
\put(7,0){\circle*{8}}\put(34,0){\circle{8}}\put(61,0){\circle*{8}}
\put(11,0){\line(1,0){19}}\put(32,0){\vertexbbbn}
\put(7,10){\makebox(0,0){$1$}}\put(34,10){\makebox(0,0){$2$}}\put(61,10){\makebox(0,0){$3$}}
  \end{picture}
  \vspace{0.5cm}
  
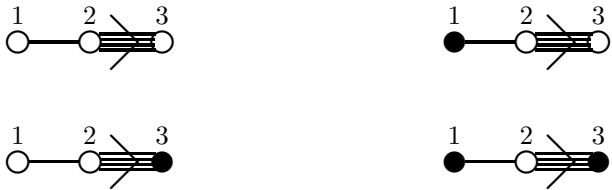
The Vogan diagrams of  $H'A_2^{(2)}$
\begin{prop}
 The equivalence clasess of Vogan diagrams on rank 3 hyperbolic Kac-Moody algebra are given by
 \begin{itemize}
  \item [(a)] $1\sim 2$
  \item [(b)] $3$
 \end{itemize}

\end{prop}
\begin{proof}
 \item [(a)] $1\sim(1,2)\sim2$ by $F\langle 1,2,2\rangle$
 \item [(b)] $3\sim(2,3)\sim(1,2,3)\sim(1,3)$ by $F\langle 3,2,1\rangle$
\end{proof}

\vspace{0.5cm}
  \begin{picture}(160,30) \thicklines
\put(7,0){\circle{8}}\put(34,0){\circle{8}}\put(61,0){\circle{8}}
\put(11,0){\line(1,0){19}}\put(32,0){\vertexcccn}
\put(7,10){\makebox(0,0){$1$}}\put(34,10){\makebox(0,0){$2$}}\put(61,10){\makebox(0,0){$3$}}
  \end{picture}
  
 \vspace{0.5cm}
  \begin{picture}(160,30) \thicklines
\put(7,0){\circle*{8}}\put(34,0){\circle{8}}\put(61,0){\circle{8}}
\put(11,0){\line(1,0){19}}\put(32,0){\vertexcccn}
\put(7,10){\makebox(0,0){$1$}}\put(34,10){\makebox(0,0){$2$}}\put(61,10){\makebox(0,0){$3$}}
  \end{picture} 
  
  \vspace{0.5cm}
  \begin{picture}(160,30) \thicklines
\put(7,0){\circle{8}}\put(34,0){\circle{8}}\put(61,0){\circle*{8}}
\put(11,0){\line(1,0){19}}\put(32,0){\vertexcccn}
\put(7,10){\makebox(0,0){$1$}}\put(34,10){\makebox(0,0){$2$}}\put(61,10){\makebox(0,0){$3$}}
  \end{picture}
  
\vspace{0.5cm}
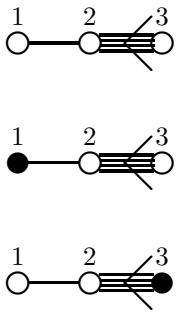
The Vogan diagrams of  $HC_{n}^{1}$\\
\begin{prop}The equivalence classes of this rank 6 hyperbolic Kac-Moody
algebra are given by
\begin{itemize}
\item  [(a)] $1\sim 2$ 
 \item [(b)] $3\sim 5$
 \item [(c)] $4$
 \item [(d)] $5$
\end{itemize}\end{prop}
\begin{proof}
 \begin{itemize}
 \item  [(a)] $1\sim (1,2)\sim(1,2,3)$ by $F\langle 1,2\rangle$ and $2\sim (1,2,3) $. So $1\sim 2$
 \item  [(b)] $3\sim (3,4)\sim (4,5)\sim 5$ by $F\langle 2,3,4,5\rangle$
 \item  [(c)] $4\sim (3,4,5)\sim (3,5)$ by $F\langle 4,5\rangle$
 \item  [(d)] $5\sim (4,5)\sim (5,6)$ by $F\langle 6,5\rangle$
 \end{itemize}
 \end{proof}

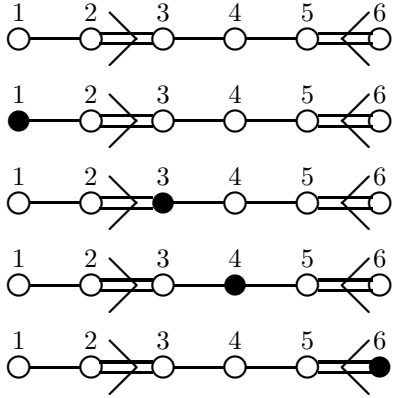
 \begin{picture}(60,30) \thicklines
\put(7,0){\circle{8}}\put(34,0){\circle{8}}\put(61,0){\circle{8}}\put(88,0){\circle{8}}\put(115,0){\circle{8}}\put(142,0){\circle{8}}
\put(11,0){\line(1,0){19}}\put(31,0){\vertexbn}\put(65,0){\line(1,0){19}}\put(92,0){\line(1,0){19}}\put(113,0){\vertexcn}
\put(7,10){\makebox(0,0){$1$}}\put(34,10){\makebox(0,0){$2$}}\put(61,10){\makebox(0,0){$3$}}
\put(88,10){\makebox(0,0){$4$}}\put(115,10){\makebox(0,0){$5$}}\put(142,10){\makebox(0,0){$6$}}
\end{picture}  

 \begin{picture}(60,30) \thicklines
\put(7,0){\circle*{8}}\put(34,0){\circle{8}}\put(61,0){\circle{8}}\put(88,0){\circle{8}}\put(115,0){\circle{8}}\put(142,0){\circle{8}}
\put(11,0){\line(1,0){19}}\put(31,0){\vertexbn}\put(65,0){\line(1,0){19}}\put(92,0){\line(1,0){19}}\put(113,0){\vertexcn}
\put(7,10){\makebox(0,0){$1$}}\put(34,10){\makebox(0,0){$2$}}\put(61,10){\makebox(0,0){$3$}}
\put(88,10){\makebox(0,0){$4$}}\put(115,10){\makebox(0,0){$5$}}\put(142,10){\makebox(0,0){$6$}}
\end{picture}  
  
 \begin{picture}(60,30) \thicklines
\put(7,0){\circle{8}}\put(34,0){\circle{8}}\put(61,0){\circle*{8}}\put(88,0){\circle{8}}\put(115,0){\circle{8}}\put(142,0){\circle{8}}
\put(11,0){\line(1,0){19}}\put(31,0){\vertexbn}\put(65,0){\line(1,0){19}}\put(92,0){\line(1,0){19}}\put(113,0){\vertexcn}
\put(7,10){\makebox(0,0){$1$}}\put(34,10){\makebox(0,0){$2$}}\put(61,10){\makebox(0,0){$3$}}
\put(88,10){\makebox(0,0){$4$}}\put(115,10){\makebox(0,0){$5$}}\put(142,10){\makebox(0,0){$6$}}
\end{picture}  

 \begin{picture}(60,30) \thicklines
\put(7,0){\circle{8}}\put(34,0){\circle{8}}\put(61,0){\circle{8}}\put(88,0){\circle*{8}}\put(115,0){\circle{8}}\put(142,0){\circle{8}}
\put(11,0){\line(1,0){19}}\put(31,0){\vertexbn}\put(65,0){\line(1,0){19}}\put(92,0){\line(1,0){19}}\put(113,0){\vertexcn}
\put(7,10){\makebox(0,0){$1$}}\put(34,10){\makebox(0,0){$2$}}\put(61,10){\makebox(0,0){$3$}}
\put(88,10){\makebox(0,0){$4$}}\put(115,10){\makebox(0,0){$5$}}\put(142,10){\makebox(0,0){$6$}}
\end{picture}

 \begin{picture}(60,30) \thicklines
\put(7,0){\circle{8}}\put(34,0){\circle{8}}\put(61,0){\circle{8}}\put(88,0){\circle{8}}\put(115,0){\circle{8}}\put(142,0){\circle*{8}}
\put(11,0){\line(1,0){19}}\put(31,0){\vertexbn}\put(65,0){\line(1,0){19}}\put(92,0){\line(1,0){19}}\put(113,0){\vertexcn}
\put(7,10){\makebox(0,0){$1$}}\put(34,10){\makebox(0,0){$2$}}\put(61,10){\makebox(0,0){$3$}}
\put(88,10){\makebox(0,0){$4$}}\put(115,10){\makebox(0,0){$5$}}\put(142,10){\makebox(0,0){$6$}}
\end{picture}  
\vspace{1cm}\\
Note- We have drawn only single painted equivalance class, one can obtained double or triple painted
   Vogan diagrams non equivalent to single painted diagram. \\
The Vogan diagrams of $HD_{n,n+2}^{(1)}$

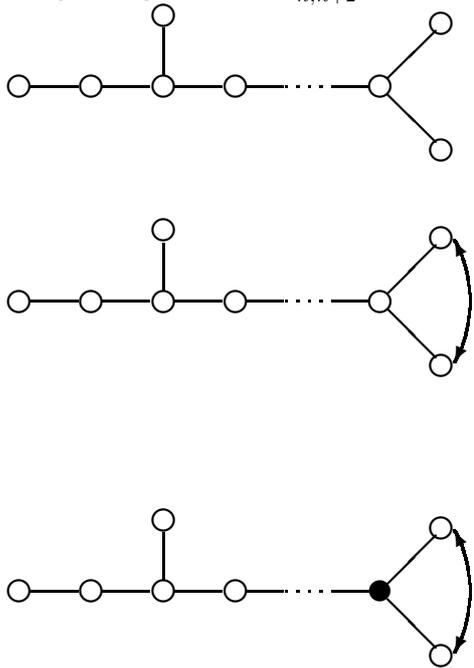
\begin{picture}(160,30) \thicklines
   \put(7,0){\circle{8}}\put(34,0){\circle{8}}\put(61,27){\circle{8}}\put(61,0){\circle{8}}
   \put(88,0){\circle{8}}\put(142,0){\circle{8}}\put(165,24){\circle{8}}\put(165,-24){\circle{8}}
  \put(11,0){\line(1,0){18}}\put(38,0){\line(1,0){18}}\put(65,0){\line(1,0){18}}\put(61,4){\line(0,1){18}}\put(92,0){\line(1,0){14}}
  \put(106,0){\dottedline{4}(1,0)(14,0)}\put(124,0){\line(1,0){14}}\put(145,3){\line(1,1){18}}\put(145,-3){\line(1,-1){18}}
        \end{picture} 
  
 \begin{picture}(160,80) \thicklines
   \put(7,0){\circle{8}}\put(34,0){\circle{8}}\put(61,27){\circle{8}}\put(61,0){\circle{8}}
   \put(88,0){\circle{8}}\put(142,0){\circle{8}}\put(165,24){\circle{8}}\put(165,-24){\circle{8}}
  \put(11,0){\line(1,0){18}}\put(38,0){\line(1,0){18}}\put(65,0){\line(1,0){18}}\put(61,4){\line(0,1){18}}\put(92,0){\line(1,0){14}}
  \put(106,0){\dottedline{4}(1,0)(14,0)}\put(124,0){\line(1,0){14}}\put(145,3){\line(1,1){18}}\put(145,-3){\line(1,-1){18}}
      \qbezier(176, 0)(175, -15)(170, -23)
    \qbezier(176, 0)(175, 15)(170, 23)
    \put(170,23){\vector( -1, 2){0}}
\put(170,-23){\vector( -1, -2){0}}
  \end{picture} 
  \vspace{1cm}
  
  \begin{picture}(160,80) \thicklines
   \put(7,0){\circle{8}}\put(34,0){\circle{8}}\put(61,27){\circle{8}}\put(61,0){\circle{8}}
   \put(88,0){\circle{8}}\put(142,0){\circle*{8}}\put(165,24){\circle{8}}\put(165,-24){\circle{8}}
  \put(11,0){\line(1,0){18}}\put(38,0){\line(1,0){18}}\put(65,0){\line(1,0){18}}\put(61,4){\line(0,1){18}}\put(92,0){\line(1,0){14}}
  \put(106,0){\dottedline{4}(1,0)(14,0)}\put(124,0){\line(1,0){14}}\put(145,3){\line(1,1){18}}\put(145,-3){\line(1,-1){18}}
      \qbezier(176, 0)(175, -15)(170, -23)
    \qbezier(176, 0)(175, 15)(170, 23)
    \put(170,23){\vector( -1, 2){0}}
\put(170,-23){\vector( -1, -2){0}}

  \end{picture} 
  \vspace{1cm}
  
  Vogan diagrams of $H^{2}D_{4}^{(1)}$\\
  
\begin{prop} The equivalence classes of this rank 6 hyperbolic Kac-Moody
algebra are given by
\begin{itemize}
 \item [(a)] $1\sim2$ 
 \item [(b)] $4\sim5\sim6$
 \item [(c)] $3$
\end{itemize}
\end{prop}
\begin{proof}
 \item [(a)] $1\sim2$ by $F\langle 1,2\rangle$
 \item [(b)] From the symmetry of the diagram we get $4\sim 5\sim 6$
 \item [(c)] $3\sim (2,3,4,5,6)$ by $F\langle 3\rangle$
\end{proof}

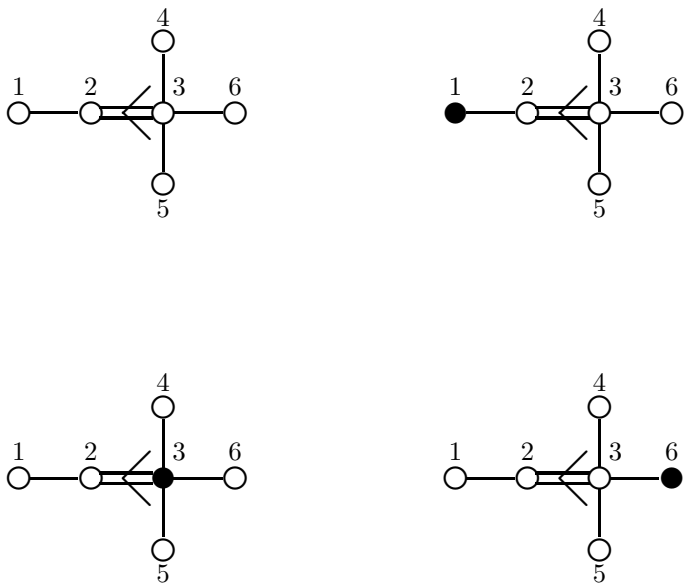
  \begin{picture}(160,80) \thicklines
     \put(7,0){\circle{8}}\put(34,0){\circle{8}}\put(61,27){\circle{8}}\put(61,0){\circle{8}}\put(61,-27){\circle{8}}\put(88,0){\circle{8}}
     \put(11,0){\line(1,0){18}}\put(31,0){\vertexcn}\put(65,0){\line(1,0){18}}\put(61,4){\line(0,1){18}}\put(61,-4){\line(0,-1){18}}
         \put(7,10){\makebox(0,0){$1$}}\put(34,10){\makebox(0,0){$2$}}\put(61,36){\makebox(0,0){$4$}}\put(61,-36){\makebox(0,0){$5$}}
     \put(67,10){\makebox(0,0){$3$}}\put(88,10){\makebox(0,0){$6$}}
   \end{picture} 
  \vspace{1cm}
  \begin{picture}(160,80) \thicklines
     \put(7,0){\circle*{8}}\put(34,0){\circle{8}}\put(61,27){\circle{8}}\put(61,0){\circle{8}}\put(61,-27){\circle{8}}\put(88,0){\circle{8}}
     \put(11,0){\line(1,0){18}}\put(31,0){\vertexcn}\put(65,0){\line(1,0){18}}\put(61,4){\line(0,1){18}}\put(61,-4){\line(0,-1){18}}
     \put(7,10){\makebox(0,0){$1$}}\put(34,10){\makebox(0,0){$2$}}\put(61,36){\makebox(0,0){$4$}}\put(61,-36){\makebox(0,0){$5$}}
     \put(67,10){\makebox(0,0){$3$}}\put(88,10){\makebox(0,0){$6$}}
   \end{picture} 
  \vspace{1cm}

  \begin{picture}(160,80) \thicklines
     \put(7,0){\circle{8}}\put(34,0){\circle{8}}\put(61,27){\circle{8}}\put(61,0){\circle*{8}}\put(61,-27){\circle{8}}\put(88,0){\circle{8}}
     \put(11,0){\line(1,0){18}}\put(31,0){\vertexcn}\put(65,0){\line(1,0){18}}\put(61,4){\line(0,1){18}}\put(61,-4){\line(0,-1){18}}
          \put(7,10){\makebox(0,0){$1$}}\put(34,10){\makebox(0,0){$2$}}\put(61,36){\makebox(0,0){$4$}}\put(61,-36){\makebox(0,0){$5$}}
     \put(67,10){\makebox(0,0){$3$}}\put(88,10){\makebox(0,0){$6$}}
   \end{picture} 
  \vspace{1cm}
  \begin{picture}(160,80) \thicklines
     \put(7,0){\circle{8}}\put(34,0){\circle{8}}\put(61,27){\circle{8}}\put(61,0){\circle{8}}\put(61,-27){\circle{8}}
     \put(88,0){\circle*{8}}
     \put(11,0){\line(1,0){18}}\put(31,0){\vertexcn}\put(65,0){\line(1,0){18}}\put(61,4){\line(0,1){18}}\put(61,-4){\line(0,-1){18}}
          \put(7,10){\makebox(0,0){$1$}}\put(34,10){\makebox(0,0){$2$}}\put(61,36){\makebox(0,0){$4$}}\put(61,-36){\makebox(0,0){$5$}}
     \put(67,10){\makebox(0,0){$3$}}\put(88,10){\makebox(0,0){$6$}}
   \end{picture} \\

  \vspace{1cm}
  
  \begin{rmk}
  We have drawn in some cases only single painted equivalance class, one can obtained double or triple painted
   Vogan diagrams which are non equivalent to single painted diagram and get rest of the Vogan diagrams of hyperbolic
   Kac-Moody algebras, since it has a very long list of similar  diagrams.
  \end{rmk}


\noindent {\bf Acknowledgement}\\

The authors thank National Board of Higher Mathematics, India (Project Grant No. 48/3/2008-R$\&$D II/196-R)
for financial support.

\begin{center}
 -xox-
\end{center}
\address{\begin{center}
Biswajit Ransingh, Institute of Mathematical Sciences, Chennai (India), email- bransingh@gmail.com 
\end{center}}
\address{\begin{center}
K C Pati, Department of Mathematics, National Institute of Technology, Rourkela (India), email- kcpati@rediffmail.com 
\end{center}}
\end{document}